\documentclass[11pt,a4paper]{article}

\usepackage[english]{babel}
\usepackage[T1]{fontenc}
\usepackage{bbm}
\usepackage{relsize}
\usepackage{amsmath, amsthm, amssymb,amsfonts}
\usepackage{mathabx} 
\usepackage{mathrsfs}
\usepackage{url}
\usepackage{shuffle}
\usepackage[dvipsnames]{xcolor}
\usepackage[linecolor=white,backgroundcolor=white,bordercolor=white,textsize=tiny]{todonotes}
\usepackage[colorlinks,citecolor=red!40!gray,linkcolor=blue!40!gray,urlcolor=black]{hyperref}
\usepackage{longtable}
\usepackage{enumitem}
\usepackage{ytableau}

\usepackage[top=1.25in,bottom=1.5in,left=1.75cm,right=1.75cm,includehead]{geometry}

\theoremstyle{plain}
\newtheorem{theorem}{Theorem}[section]
\newtheorem{proposition}[theorem]{Proposition}
\newtheorem*{proposition*}{Proposition}
\newtheorem*{theorem*}{Theorem}
\newtheorem{lemma}[theorem]{Lemma}

\newtheorem{corollary}[theorem]{Corollary}

\newtheorem{example}[theorem]{Example}
\newtheorem{notation}[theorem]{Notation}
\newtheorem*{notation*}{Notation}

\theoremstyle{definition} 
\newtheorem{definition}[theorem]{Definition}
\newtheorem*{tata}{Generalization}
  {\begin{mdframed}[backgroundcolor=lightgray]\begin{tata}}%
  {\end{tata}\end{mdframed}}
  
\makeatletter
\newtheorem*{rep@theorem}{\rep@title}
\newcommand{\newreptheorem}[2]{%
\newenvironment{rep#1}[1]{%
 \def\rep@title{#2~\ref{##1}}%
 \begin{rep@theorem}}%
 {\end{rep@theorem}}}
\makeatother
\newreptheorem{theorem}{Theorem}

\newcommand{\R}{\mathbb{R}}

\newcommand{\mute}[1]{}

\def\li{\w{a}}

\newcommand\dm{\w{d_{-1}}} 

\newcommand{\id}{\operatorname{id}}

\newcommand{\GL}{\operatorname{GL}}

\def\word#1{{\color{Blue}\mathbf{#1}}}
\newcommand\w[1]{\word{#1}}
\newcommand\emptyWord{{\color{Blue}\boldsymbol{\varepsilon}}}
\newcommand\ww[1]{{\color{Magenta}\mathbf{#1}}}

\newcommand\sign{\operatorname{sign}}

\newcommand\SL{\operatorname{SL}}
\newcommand\hs{\succ}

\newcommand\Pf{{\operatorname{Pf}}}
\newcommand\inv{\operatorname{inv}}

\DeclareFontFamily{U}{niceshuffle}{}
\DeclareFontShape{U}{niceshuffle}{m}{n}{
  <5-8> s*[2.5] shuffle7
  <8->  s*[2.5] shuffle10
}{}
\DeclareSymbolFont{NiceShuffle}{U}{niceshuffle}{m}{n}
\DeclareMathSymbol\niceshuffle{\mathop}{NiceShuffle}{"001}

\newdimen\squaresize \squaresize=20pt
\newdimen\thickness \thickness=0.4pt

\def\square#1{\hbox{\vrule width \thickness
     \vbox to \squaresize{\hrule height \thickness\vss
        \hbox to \squaresize{\hss#1\hss}
     \vss\hrule height\thickness}
\unskip\vrule width \thickness}
\kern-\thickness}

\def\vsquare#1{\vbox{\square{$#1$}}\kern-\thickness}

\def\thisbox#1{\kern-.09ex\fbox{#1}}
\def\downbox#1{\lower1.200em\hbox{#1}}

\newcommand\Wd{\mathcal W_d}
\newcommand\WTwo{\mathcal W_2}
\newcommand\WdBig{
\begin{pmatrix}
  \word{11} & \dots & \word{1d} \\
      \vdots & \ddots & \vdots     \\
  \word{d1} & \dots & \word{dd}
\end{pmatrix} }

\newcommand\tOned{\mathfrak t_{1,d}}
\newcommand\tOnedBig{
  \begin{gathered}
    \begin{ytableau}
      1  \\
      2 \\
      \vdots\\
      {d} 
   \end{ytableau}
 \end{gathered}}
\newcommand\tTwod{\mathfrak t_{2,d}}
\newcommand\tTwodBig{
      \begin{gathered}
      \begin{ytableau}
    1 & 2 \\
    3 & 4 \\
    \vdots & \vdots \\
    \scalebox{0.4}{ ${2d-1}$ } & {2d}
  \end{ytableau}
      \end{gathered}
}

\newcommand{\operator}[1]{\mathop{\vphantom{\sum}\mathchoice
{\vcenter{\hbox{$#1$}}}
{\vcenter{\hbox{$#1$}}}{#1}{#1}}\displaylimits}
\newcommand{\aprod}{\operator{\overrightarrow{\prod}}}

\begin{document}

\title{A quadratic identity in the shuffle algebra and \\ an alternative proof for de Bruijn's formula}
\date{\today}
\author{Laura Colmenarejo\footnote{L. Colmenarejo, University of Massachusetts at Amherst (USA)} \and Joscha Diehl\footnote{J. Diehl, University of Greifswald (Germany)} \and Miruna-\c Stefana Sorea\footnote{M.-\c S. Sorea, SISSA, Trieste (Italy)}}

\maketitle

\begin{abstract}
Motivated by a polynomial identity of certain iterated integrals, 
first observed in~\cite{bib:colmen1} in the setting of lattice paths,
we prove an intriguing combinatorial identity in the shuffle algebra.
It has a close connection to \emph{de Bruijn's formula} when
interpreted in the framework of signatures of paths.
\end{abstract}

\section{Introduction}

A \emph{path} is a continuous map $X : [0,1]\rightarrow\mathbb{R}^d$. We shall assume that the components of $X$, $X^{(i)}$, for $i=1,\ldots,d$, are (piecewise) continuously differentiable functions.
Describing phenomena parametrized by time,
they appear in many branches of science such as mathematics, physics, medicine or finance. For mathematics in particular, see~\cite{bib:colmen1} for references. 

One way to look at paths is through their \emph{iterated integrals}
\begin{align*}
  \int dX^{({w_1})}_{t_1} \dots dX^{({w_n})}_{t_n}
  :=
  \int_{0 < t_1 < \dots < t_n < 1} \dot X^{({w_1})}_{t_1} \dots \dot X^{({w_n})}_{t_n} dt_1 \dots dt_n,
\end{align*}
where $n \ge 1$ and ${w_1}, \dots, {w_n} \in \{1, \dots, d\}$. Here we use the abbreviated notations: $X_{t_i}:=X(t_i),$ with $t_i\in [0,1],$ and $\dot X^{(w_i)}_{t_i}$ denotes the derivative of $X^{(w_i)}$ with respect to the variable $t_i$. 
The first systematic study of these integrals was
undertaken by Kuo Tsai Chen~\cite{bib:chen}. For example, Chen proved that, up to an equivalence relation, iterated integrals uniquely determine a path.

In the field of stochastic analysis paths are usually \emph{not differentiable}.
Nonetheless (stochastic) integrals have played a major role there
and this culminated in Terry Lyons' theory of \emph{rough paths}~\cite{bib:LQ,bib:FH}.
Owing to their descriptive power of nonlinear phenomena (compare Chen's uniqueness result above),
these objects have recently been successfully applied to statistical learning
\cite{bib:lyons2014feature,bib:kormilitzin2017detecting,bib:KO2019,bib:li2019skeleton}.

Iterated integrals possess an interesting Hopf algebraic structure~\cite{bib:reut}.
More recently they have been studied from the viewpoint of
(applied) algebraic geometry~\cite{bib:amnd,bib:galup,bib:stuToric,bib:pfeff} and representation theory~\cite{bib:DR2018}. 

In~\cite{bib:colmen1} an approach of studying signature-path tensors through the lens of toric geometry is presented and  
the following curious property was discovered.
\begin{theorem}[{\cite[Corollary 3.23]{bib:colmen1}}]\label{thm:CGM}
  \begin{equation}
    \label{eq:CGM}
    \det
    \begin{pmatrix}
      \displaystyle{\int} dX^{(1)} dX^{(1)} & \dots & \displaystyle{\int} dX^{(1)} dX^{(d)} \\
             \vdots           & \ddots & \vdots \\
      \displaystyle{\int} dX^{(d)} dX^{(1)} & \dots & \displaystyle{\int} dX^{(d)} dX^{(d)}
    \end{pmatrix}  = \frac{1}{2^d} \left( \sum_{\sigma \in \mathrm S_d} \sign(\sigma) \displaystyle{\int} dX^{\sigma(1)} \dots dX^{\sigma(d)} \right)^2.
  \end{equation}
\end{theorem}
Their proof is based on calculations with lattice paths.
One consequence of this result is that
it gives an \emph{inequality} for the
iterated integrals of order two (i.e. this particular
determinant is non-negative).
Notice that when working on the Zariski closure of the space of signatures, as in the framework of~\cite{bib:amnd}, one can only obtain \emph{equalities}.
See~\cite[p. 22-23]{bib:colmen1} for more details.


In the current paper we look at this statement from a purely algebraic perspective. (We advise the reader to peek ahead to Section~\ref{sec:prelim} for notation used in the following.)
Consider the space of formal linear combinations of words in the alphabet $[\w{d}]=\{\w{1},\dots,\w{d}\}$, and denote it by $T(\R^d)$.
Given two words $\w{w}$ and $\w{v}$ in the alphabet $[\w{d}]$, the \emph{shuffle product} of them, $\w{w}\shuffle \w{v}$, is defined as the sum of all permutations of the concatenated word $\w{w}\w{v}$, that keep the respective order of the two words intact. For example,
\begin{align*}
  \w{12} \shuffle \w{34}
  =
  \w{1234} + \w{1324} + \w{1342} + \w{3124} + \w{3142} + \w{3412}.
\end{align*}
{Then, $\left(T(\R^d),\shuffle\right)$ is an algebra known as the \emph{shuffle algebra}.} In a sense made precise in Section~\ref{sec:prelim}
the determinant (in the shuffle algebra) of the matrix
\begin{align*}
  \WdBig
\end{align*}
relates to the determinant (in $\R$) appearing in~\eqref{eq:CGM}.
To reformulate the other side of the equality in~\eqref{eq:CGM},
we recall the basis of $\SL(\R^d)$-invariants in $T(\R^d)$, going back to Weyl~\cite{bib:weyl1946classical},
and presented, in the language relevant used here, in~\cite{bib:DR2018}.
This basis is indexed by \emph{standard Young tableaux} of shape $\underbrace{(w, \dots, w)}_{d \text{ times }}$, for $w\ge 1$. For us, only the cases $w=1,2$ are relevant and we denote the basis by $\{\inv(T)\}$, where $T$ ranges over these standard Young tableaux. The definition of this basis is presented in Section~\ref{subsec:volume}, and we illustrate with two examples here:
\begin{align}\label{eq:invExample}
\begin{split}
    \inv\left(
        \begin{gathered}
          \begin{ytableau}
            1 \\
            2
          \end{ytableau}
        \end{gathered}
      \right) &= \w{12} - \w{21},\qquad \text{ and } \\
    \inv\left(
        \begin{gathered}
          \begin{ytableau}
            1 & 2 \\
            3 & 4 \\
          \end{ytableau}
        \end{gathered}
    \right) &= \w{1122} - \w{1221} - \w{2112} + \w{2211}.
\end{split}
\end{align}

We can now state our main theorem {(see Theorem~\ref{thm:main} in Section~\ref{sec:main})}. 
\begin{theorem*}
  For $d\geq 1$ the following two equalities hold:
  \begin{align*}
    {\det}_\shuffle \WdBig
    =
    \inv\left( \tTwodBig \right)
    =
    \frac{1}{2^d}
    \inv\left( \tOnedBig \right)^{\shuffle 2}.
  \end{align*}
\end{theorem*}
In particular, we give a new and purely algebraic proof of the fact that the determinant of the second level of the iterated integrals signature is a square, Theorem~\ref{thm:CGM} (see Corollary~\ref{cor:relationToCGM}).
As an application of our approach, we obtain the \emph{de Bruijn's formula} for the Pfaffian.
We note that de Bruijn's formula was looked at in the language of shuffle algebras already in~\cite{bib:LuqueThibon}.
Our viewpoint differs in that we approach the topic through invariant theory.

In Section~\ref{sec:prelim}, we present the framework of the shuffle algebra and the half-shuffle products, together with their relation with the signatures of paths. Since our proof for the main theorem uses some results from representation theory, we include in Section~\ref{subsec:RTsymmetricgroup} the necessary basic notions. Moreover, in Section~\ref{subsec:volume} we relate our setting with the work of~\cite{bib:DR2018} about $\SL(\R^d)$-invariants and signed volume. 
In Section~\ref{sec:aux} we present already known results that will help us to prove our main results. For instance, we present several identities related to the shuffle product that already appeared in~\cite{bib:And1883}. 
Section~\ref{sec:main} is dedicated to prove our main theorem. In Section~\ref{subsec:deBruijn}, we present the relation of our main theorem with the de Bruijn's formula for both the even and odd-dimensional cases.

\subsection*{Acknowledgements}
The authors express their gratitude to Bernd Sturmfels and Mateusz Micha\l{}ek for helpful discussions and suggestions, and for bringing the team together. The authors are also grateful to the MPI MiS for giving us the opportunity to work on this project in such a nice environment. 
We would like to thank Darij Grinberg for making us aware of several inconsistencies in an earlier version. We are also grateful to the anonymous referees, for their careful reading and for their comments that helped us improve the quality of our manuscript.

\subsection*{Funding}
Laura Colmenarejo was partially supported by MTM2016-75024-P.

\section{Preliminaries}\label{sec:prelim}

In this section we present several connected frameworks. First, we talk about
the shuffle algebra and the half-shuffle operation. A survey on the history of
the shuffle product can be found in~\cite{bib:FP} and our presentation here is inspired
by~\cite{bib:colmenPreis}, especially the part concerning the half-shuffle
operation. Next, we describe briefly the relation between the setting of
signatures of paths and the shuffle algebra. We refer the reader to
\cite[Subsection 2.1]{bib:DR2018} for more details. We also include a brief
summary of the representation theory setting, which can
be found in~\cite{bib:FuHa} with full detail. We finish by presenting a brief
overview of the $\SL(\R^d)$-invariants inside of $T(\R^d)$.

\subsection{The shuffle algebra and half--shuffle products}

Let us consider the alphabet $[\w{d}]=\{\w{1},\dots,\w{d}\}$
and denote by $[\w{d}]^*$ the set of words of any length in this alphabet, including the empty word $\emptyWord$.
Denote by $\mathrm{T}(\mathbb{R}^d)$ the space of \emph{finite linear combinations} of these words. Together with the concatenation product, $\w{w}\cdot \w{v}$, $\left(\mathrm{T}(\mathbb{R}^d),\cdot\right)$ is an algebra, known as the \emph{tensor algebra}.

Let us consider another operation on $\mathrm{T}(\mathbb{R}^d)$. 
\begin{definition}\label{def:shuffleRecursuve}
Consider the words $\w{u}$, $\w{v}$ and $\w{w}$, and the letters $\w{a}$ and $\w{b}$. The \emph{shuffle product} of two words is defined recursively by
\begin{align*}
  \emptyWord \shuffle \w{u} &= \w{u} \shuffle \emptyWord = \w{u},\qquad \text{ and }\\
  (\w{v} \cdot \w{a}) \shuffle (\w{w}\cdot \w{b}) &= \left(\w{v}\shuffle (\w{w}\cdot \w{b})\right)\cdot \w{a} + \left((\w{v}\cdot \w{a} )\shuffle \w{w}\right)\cdot \w{b}.
\end{align*}
This operation extends bilinearily to a commutative product on all of $T(\R^d)$. The algebra $\left(\mathrm{T}(\mathbb{R}^d), \shuffle\right)$ is known as the \emph{shuffle algebra}.
\end{definition}

The shuffle product can be seen as the symmetrisation of the right half-shuffle product.
\begin{definition}\label{def:half-shuffle}
The \emph{right half-shuffle product}
is recursively given on words as
\begin{equation*}
 \w{w}\hs\li := \w{w}\li, \qquad \text{and} \qquad
 \w{w}\hs \w{v}\li := (\w{w}\hs \w{v}+\w{v}\hs \w{w})\cdot \li,
\end{equation*}
where $\w{w},\w{v}$ are words and $\li$ is a letter.
\end{definition}
The definition of the half-shuffle goes back to considerations in topology \cite[Section 18]{bib:EM}
and in Lie theory \cite{schutzenberger1958propriete}.
Many variants and generalizations are known, see for example \cite{aguiar2004quadri}, \cite{burgunder2010tridendriform}.

Note that if $\w{w},\w{v}$ are any non-empty words, then $\w{w}\shuffle \w{v} = \w{w}\hs \w{v} + \w{v}\hs\w{w}$.
Hence in Definition~\ref{def:half-shuffle} we can replace $ \w{w}\hs \w{v}\li = (\w{w}\hs \w{v}+\w{v}\hs \w{w})\cdot \li$ with the equivalent equality $\w{w}\hs \w{v} \li= ( \w{w}\shuffle \w{v})\cdot \li$. In general, for non-empty words, one has 
\begin{align}
\label{eq:dendriform}
    \w{u}\hs(\w{v}\hs \w{w})=(\w{u}\shuffle \w{v})\hs \w{w}.
\end{align}
In particular, notice that the shuffle product is associative, while the half-shuffle product is not.

\subsection{Iterated-integrals signatures of paths}

For a (piecewise) smooth path $X: [0,1] \to \R^d$ the collection of \emph{iterated integrals}
\begin{align*}
  \int dX^{(\w{a_1})}_{t_1} \dots dX^{(\w{a_n})}_{t_n}
  :=
  \int_{0 < t_1 < \dots < t_n < 1} \dot X^{(\w{a_1})}_{t_1} \dots \dot X^{(\w{a_n})}_{t_n} dt_1 \dots dt_n,
\end{align*}
where $n \ge 1$ and $\w{a_1}, \dots, \w{a_n} \in [\w{d}]$, is conveniently stored
in the \emph{iterated-integrals signature}
\begin{align*}
  \sigma(X)
  :=
  \sum_{\substack{n \ge 0, \\ \w{w}=\w{a_1}\dots \w{a_n} \in [\w{d}]^*}}
  \int dX^{(\w{a_1})}_{t_1} \dots dX^{(\w{a_n})}_{t_n}
  \quad \w{w}.
\end{align*}
This object is a \emph{formal infinite sum} of words $\w{w}$ in the alphabet
$[\w{d}]$ whose coefficients are given by the integrals. 

Alternatively, $\sigma(X)$ can be seen as a \emph{linear function} on $T(\R^d)$ given by 
\begin{align*}
  \Big\langle \w{w}, \sigma(X) \Big\rangle
  =
  \int dX^{(\w{w_1})}_{t_1} \dots dX^{(\w{w_n})}_{t_n},
\end{align*}
for any element $\w{w} \in [\w{d}]^*$, and extended linearly to all of  $T(\R^d)$.
It turns out that $\sigma(X)$ is in fact a \emph{multiplicative
character} (i.e. an algebra morphism into $\R$) on $(T(\R^d),\shuffle)$, see
\cite[Corollary 3.5]{bib:reut}. 
This fact is also called the \emph{shuffle identity} and reads as
\begin{align}
  \label{eq:shuffleIdentity}
  \Big\langle \w{w} \shuffle \w{v}, \sigma(X) \Big\rangle = \Big\langle \w{w}, \sigma(X) \Big\rangle \cdot \Big\langle \w{v}, \sigma(X) \Big\rangle, \qquad \forall \w{w},\w{v} \in T(\R^d).
\end{align}


\subsection{Representation theory}\label{subsec:RTsymmetricgroup}
This section provides some broad ideas from representation theory that are used in this paper.
For more details, we refer the reader to~\cite{bib:FuHa}.

Given a group $G$ and an $n$-dimensional vector space over $\R$, $V$, we say that the map $\rho: G \longrightarrow \GL(V))$ is a \emph{representation} of $G$ if $\rho$ is a group homomorphism, where $\GL(V)$ is the group of automorphisms of $V$. In general, $\rho$ is identified with $V$, and $\rho(g)(v) := g\cdot v$, for all $v\in V$ and $g\in G$. In this sense, we say that the group acts (on the left) on the vector space. There are different ways to construct representations from other representations; for instance, by constructing the direct sum or the tensor product of representations, or restricting or inducing representations from groups to subgroups, and vice versa.

Another approach to understand representations is by looking at them as modules,
which allows us to study representations as vector spaces over the group algebra $\R[G]$ (i.e. set of all linear combinations of elements in $G$ with coefficients in $\R$). 

A (non-zero) representation $V$ of $G$ is said to be \emph{irreducible} if the only subspaces of $V$ invariant under the action of $G$ are the vector space itself and the trivial subspace. By Schur's Lemma and Maschke's Theorem, we know that given any representation of a well-behaved\footnote{$G$ finite or $G=\SL(\R^d)$ will do.}
group
$G$ we can decompose it as the direct sum of irreducible representations.

That is $V = \bigoplus_k I_k$,
where $I_k \cong W_k^{\oplus n_k}$,
the $W_k$ are pairwise non-isomorphic irreducible representations,
and $n_k$ denotes their multiplicity (in $V$).
The $I_k$ are called the \emph{isotypic components of $V$} and this decomposition is called the \emph{canonical decomposition} of $V$ or the \emph{decomposition into isotypic components} $I_k$. Note that the decomposition into the $I_k$ is unique, while the 
decomposition of $I_k$ into the $n_k$ copies of $W_k$ is \emph{not}.%
\footnote{
The idea for the isotypic decomposition is that there may be several irreducible representations
inside of $V$ that are isomorphic, and they are 'bundled' in $I_k$.}
One of the main general goals in Representation Theory is to characterise the irreducible representations of a group and to give an algorithm to compute the canonical decomposition of an arbitrary representation.

Each representation is also characterised by the trace of the matrix associated to each element of the group. This vector is known as the \emph{character} of the representation, and it is invariant under conjugation. Character theory is closely related to the theory of symmetric functions. 

In this section we focus our attention on the symmetric group and on the general linear group. Let $\mathrm S_d$ be the group of permutations of $\{1,2,\dots,d\}$, for $d\geq 1$. In the case of the symmetric group $\mathrm S_d$, the conjugacy classes are in bijection with the partitions of $d$, which are weakly decreasing sequences of positive integers that sum up to $d$. This bijection is given by the decomposition of the permutations in cycles. 
Given a partition $\lambda=(\lambda_1,\lambda_2,\dots, \lambda_\ell)$, we
associate to $\lambda$ a  Young diagram, which is just an array of boxes with
$\lambda_i$ boxes in the $i^{th}$ row. Then, the standard Young tableaux are
fillings of Young diagrams with all the numbers in the set $\{1,2,\dots,d\}$
that are increasing in columns and rows.

\subsection{$\SL$ invariants and signed volume}\label{subsec:volume}

The natural representation of $\SL(\R^d)$ on $\R^d$ induces a representation on
$T(\R^d)$.
The study of invariants to this action goes back over a century, see~\cite{bib:weyl1946classical}
for a good starting point of the literature.
We recall the presentation used in~\cite{bib:DR2018}.

As mentioned in the introduction,
a basis for the invariants is indexed
by \emph{standard Young tableaux} of shape $\underbrace{(w, \dots, w)}_{d \text{ times }}$, for $w\ge 1$ arbitrary. 
Given a Young tableaux $T$, the corresponding invariant basis element, denoted $\inv( T)$\footnote{In~\cite{bib:DR2018} it is denoted by $\iota( e_T )$.}
is obtained as follows:
let $n = w d$ be the length of $T$ and consider the word $\w{w} = \w{a_1} \w{a_2} \dots \w{a_n}$, where $\w{a_\ell} = \w{i}$ if and only if $\ell$ is in the $i^{\text{th}}$ row of $T$.
Then
\begin{align*}
  \inv( T ) := \sum_\sigma \sign(\sigma) \sigma \w{w},
\end{align*}
where the sum is over all permutations $\sigma \in \mathrm S_n$ that leave the values in each column of $T$ unchanged. For example, the tableau  
$T = \begin{ytableau}
        1 & 2 \\
        3 & 4
\end{ytableau}$ gives the word $\w{1122}$. Therefore, 
\begin{equation*}
  \inv\left(
    \begin{gathered}
      \begin{ytableau}
        1 & 2 \\
        3 & 4 \\
      \end{ytableau}
    \end{gathered}
  \right)
  = (\id - (13) - (24) + (13)(24)) \w{1122} = \w{1122} - \w{2112} - \w{1221} + \w{2211}.
\end{equation*}

Our study only looks at  the following two particular standard Young tableaux for $w=1,2$
\footnote{The basis element $\inv( \tOned )$ has a nice geometric interpretation
  in the setting of iterated integrals, see~\cite{bib:DR2018} where it is denoted by $\inv_d$.}
\begin{align}
  \label{eq:twoSTYs}
  \tOned := \tOnedBig,
  \quad
  \tTwod := \tTwodBig.
\end{align}
In this case, we have that
\begin{align}
\inv\left(\mathfrak t_{1,d}\right) &= \sum_{\sigma\in \mathrm S_d} \sign(\sigma)\sigma(\w{12}\cdots\w{d}) \label{eq:inv1}\\
\inv\left(\mathfrak t_{2,d}\right) &= \sum_{\sigma\in H_d} \sign(\sigma)\sigma(\w{1122}\cdots\w{dd}),    \label{eq:inv2}
\end{align}
where $H_d := \Big\langle (1,3), (2,4), (3,5), \dots, (2d-2,2d) \Big\rangle \subseteq \mathrm S_{2d}$.

We also recall the matrix introduced before:
\begin{align}
  \label{eq:Wd}
  \Wd
  :=
  \WdBig,
\end{align}
which is an element of $A^{d\times d}$ where $A = T(\R^d)$, seen as a commutative algebra (over $\R$).

\section{Auxiliary results}\label{sec:aux}

\subsection{An identity by Andr\'eief}

\begin{notation}\label{not:hs}
Please note that we use the convention of strict left bracketings for the half-shuffle product.%
  \footnote{Recall that the half-shuffle product is non-associative, so the specification of a bracketing is necessary.}
  For words $\w{w_1},\dots,\w{w_k}$,
  $\w{w_1} \hs \w{w_2} \hs \dots \hs \w{w_k} :=
    ( (\w{w_1} \hs \w{w_2}) \hs \dots ) \hs \w{w_k}$.
\end{notation}
The following lemma expresses the shuffle product of $k$ words in terms of half-shuffle products. 
\begin{lemma}\label{lemma:inductionShHs}
  For any non-empty words $\w{w_1}, \dots, \w{w_k}$, the following equality holds:
  \begin{align}\label{eq:shHs}
    \w{w_1} \shuffle \dots \shuffle \w{w_k} =
    \sum_{\sigma \in \mathrm{S}_k} \w{w}_{\sigma (\w{1})} \hs \w{w}_{\sigma(\w{2})} \hs \dots \hs \w{w}_{\sigma (\w{k})}.
  \end{align}
\end{lemma}
\begin{proof}
The identity can be proven by a simple induction on $k$,
using \eqref{eq:dendriform}.
It also follows from a more abstract argument as follows.
First, the statement is immediately seen to be true
if the words $\w{w_1}, .. \w{w_k}$ are in fact single (distinct) letters.
Now, $T(\R^k)$ is the free commutative dendriform
algebra over $k$ letters,
\cite[p.19]{schutzenberger1958propriete},
\cite[Proposition 1.8]{loday1995cup}.
In particular, for any commutative dendriform algebra
$Z$ and any map $\phi: \{1,..,k\} \to Z$ there
exists a unique morphism $\Phi: T(\R^k) \to Z$ of commutative dendriform algebras
satisfying
\begin{align*}
    \Phi( \w{i} ) = \phi({i}),   \qquad i = 1,\dots, k.
\end{align*}
We specialize to $Z = T(\R^d)$
and $\phi(i) = \w{w_i}$.
Now, in $T(\R^k)$ we have the following identity
\begin{align*}
   \w{1} \shuffle ... \shuffle \w{k}
   =
   \sum_{\sigma\in \mathrm S_k} \sigma(\w{1}) \hs ... \hs \sigma(\w{k}).
\end{align*}
Therefore, using that $\Phi$ is a morphism, we have that
\begin{align*}
   \w{w_1} \shuffle ... \shuffle \w{w_k}
   &=
   \Phi(\w{1}) \shuffle ... \shuffle \Phi(\w{k}) =
   \Phi\left( \w{1} \shuffle ... \shuffle \w{k} \right) =
   \Phi\left( 
   \sum_{\sigma\in \mathrm S_k} \sigma(\w{1}) \hs ... \hs \sigma(\w{k}) \right) \\ &=
   \sum_{\sigma\in \mathrm S_k} \Phi(\sigma(\w{1})) \hs ... \hs \Phi(\sigma(\w{k})) =
   \sum_{\sigma \in \mathrm{S}_k} \w{w}_{\sigma (\w{1})} \hs \w{w}_{\sigma(\w{2})} \hs \dots \hs \w{w}_{\sigma (\w{k})},
\end{align*}
as desired.
\end{proof}

The following statement seems to first appear in~\cite{bib:And1883}
(see also~\cite[p.1]{bib:deBruijn} and \cite{forrester2019meet})
and is also known as the continuous (or generalized) Cauchy-Binet formula~\cite[Proposition 2.10]{johansson2005random}.
\begin{lemma}[{\cite{bib:And1883}}]
  \label{lemma:andreief}
  Let $\{\w1,\dots,\w{d}\}$ and
  $\{\ww1,\dots,\ww{d}\}$
  be two alphabets in $d$ letters each.
  Then
  \begin{align}
  \label{eq:twoAlphabets}
    {\det}_\shuffle
    \begin{pmatrix}
      \w1\ww1 & \w1\ww2 & \dots & \w1\ww{d} \\
      \w2\ww1 & \w2\ww2 & \dots & \w2\ww{d} \\
       \dots  & \dots   & \dots & \dots     \\
       \w{d}\ww1 & \w{d}\ww2 & \dots & \w{d}\ww{d}
    \end{pmatrix}
    =
    \sum_{\sigma, \tau \in \mathrm{S}_d} \sign(\sigma)\sign(\tau) \left[\sigma(\w1)\tau(\ww1) \hs \sigma(\w2)\tau(\ww2) \hs \dots \hs \sigma(\w{d})\tau(\ww{d})\right] 
  \end{align}
\end{lemma}

\begin{proof}
  For a fixed $\sigma\in\mathrm{S}_d$, applying Lemma~\ref{lemma:inductionShHs} with $\w{w_i}:=\w{i} \sigma(\ww{i})$, we get
  \begin{align*}
    \displaystyle{\niceshuffle_{i=1}^d\ \word{i} \sigma(\ww{i})}
    &=
    \sum_{\tau\in \mathrm{S}_d}\tau(\word{1})\sigma(\tau\big (\ww{1}))\hs\tau(\word{2})\sigma(\tau\big (\ww{2}))\hs\ldots\hs\tau(\word{d})\sigma(\tau(\ww{d})). 
  \end{align*}
  
  Then, the left-hand side of \eqref{eq:twoAlphabets} is equal to,
  \begin{align*}
    \sum_{\sigma\in \mathrm S_d} \sign(\sigma) \displaystyle{\niceshuffle_{i=1}^d}\ \word{i} \sigma(\ww{i})
    &=
    \sum_{\sigma \in \mathrm{S}_d}
    \sum_{\tau\in \mathrm{S}_d}
    \sign(\sigma)
     \tau(\word{1})\sigma(\tau\big (\ww{1}))\hs\tau(\word{2})\sigma(\tau\big (\ww{2}))\hs\ldots\hs\tau(\word{d})\sigma(\tau(\ww{d})) \\
    &=
    \sum_{\sigma \in \mathrm{S}_d}
    \sum_{\tau\in \mathrm{S}_d}
    \sign(\sigma\circ \tau)
    \sign(\tau)
     \tau(\word{1})\sigma(\tau\big (\ww{1}))\hs\tau(\word{2})\sigma(\tau\big (\ww{2}))\hs\ldots\hs\tau(\word{d})\sigma(\tau(\ww{d})) \\
    &=
    \sum_{\rho \in \mathrm{S}_d}
    \sum_{\tau\in \mathrm{S}_d}
    \sign(\rho)
    \sign(\tau)
     \tau(\word{1})\rho\big(\ww{1})\hs\tau(\word{2})\rho\big (\ww{2})\hs\ldots\hs\tau(\word{d})\rho(\ww{d}),
  \end{align*}
  as desired.
\end{proof}

\subsection{The shuffle determinant}
In this section we present a technical lemma that is crucial for the proof of our main result.
This lemma states that, for the determinant $\det_\shuffle (\Wd)$, the shuffle
of letters can be replaced by a ``shuffle of blocks of 2 letters''. First, let
us see one example. 
\begin{example}
For $d=2$, we have that
    \begin{align*}
    {\det}_\shuffle( \WTwo ) &= 
    {\det}_\shuffle
      \begin{pmatrix}
        \word{11} & \word{12} \\
        \word{21} & \word{22}
    \end{pmatrix}
    = \word{1122} + \word{2211} - \word{1221} - \word{2112} \\
    &=
    \text{ shuffle of the block $\word{11}$ with the block $\word{22}$} 
    -
    \text{ shuffle of the block $\word{12}$ with the block $\word{21}$}.
    \end{align*}
    
    Further note, that this is equal to
    $\inv\left(
        \begin{gathered}
          \begin{ytableau}
            1 & 2 \\
            3 & 4 \\
          \end{ytableau}
        \end{gathered}\right)$.
\end{example}
The general statement is as follows.
\begin{proposition}
  \label{prop:directProof}
  \begin{align*}
    {\det}_\shuffle( \Wd )
    = \sum_{\sigma\in \mathrm{S}_d}\sign(\sigma)\bigg (\sum_{\tau\in \mathrm{S}_d} \aprod_{i=1}^d \tau(\word{i})\sigma(\tau\big (\word{i})\big )\bigg )
    = \inv\left( \tTwodBig \right),
  \end{align*}
  where the product $\aprod$ denotes the concatenation product.
\end{proposition}

Before presenting the proof for Proposition~\ref{prop:directProof},
we illustrate its idea in the case $d=3$. 
\begin{example}\label{ex:d3}
  Consider $d=3$. By Lemma~\ref{lemma:andreief}, we obtain
  \begin{align*}
      &\sum_{\sigma\in \mathrm{S}_3}\sign(\sigma)
      \niceshuffle_{i=1}^3\ \word{i} \sigma(\word{i}) 
      =\sum_{\sigma, \tau \in \mathrm{S}_3} \sign(\sigma)\sign(\tau) \left[\sigma(\w1)\tau(\w1) \hs \sigma(\w2)\tau(\w2) \hs \sigma(\w{3})\tau(\w{3})\right]\\
      &= (\w{11} \hs \w{22}) \hs \w{33} +
    (\w{11} \hs \w{33}) \hs \w{22} +
    (\w{22} \hs \w{11}) \hs \w{33} +
    (\w{22} \hs \w{33}) \hs \w{11} \\
    &\;+(\w{33} \hs \w{11}) \hs \w{22} +
    (\w{33} \hs \w{22}) \hs \w{11} + \dots 
    - (\w{13} \hs \w{22}) \hs \w{31} - 
    (\w{13} \hs \w{31}) \hs \w{22}  \\
    &\;-(\w{22} \hs \w{13}) \hs \w{31} - 
    (\w{22} \hs \w{31}) \hs \w{13}  
    - (\w{31} \hs \w{13}) \hs \w{22}
    - (\w{31} \hs \w{22}) \hs \w{31}=:\star.
     \end{align*}
  Now observe that the term 
  $$(\w{11} \hs \w{22}) \hs \w{33} =    
  (\w{11} \hs \w{22}) \cdot \w{33}+
(\w{11} \shuffle \w{2} \shuffle \w{3}) \cdot \w{2} \w{3}$$ 
cancels with the term 
    $$(\w{11} \hs \w{32}) \hs \w{23} =
    (\w{11} \hs \w{32}) \cdot \w{23} +
    (\w{11} \shuffle \w3 \shuffle \w2) \cdot \w{23},$$ 
  to give the term $(\w{11} \hs \w{22}) \cdot \w{33}
    -(\w{11} \hs \w{32}) \cdot \w{23}$.

  Applying this to all terms, we replace the right-most half-shuffle product by a concatenation product. That is,
  \begin{align*}
    \star =&
    (\w{11} \hs \w{22}) \cdot \w{33}
    +
    (\w{11} \hs \w{33}) \cdot \w{22}
    +
    (\w{22} \hs \w{11}) \cdot \w{33}
    +
    (\w{22} \hs \w{33}) \cdot \w{11} 
    +
    (\w{33} \hs \w{11}) \cdot \w{22}
    +
    (\w{33} \hs \w{22}) \cdot \w{11}  +
    \dots \\
    &
    - (\w{13} \hs \w{22}) \cdot \w{31}
    - (\w{13} \hs \w{31}) \cdot \w{22}
    - (\w{22} \hs \w{13}) \cdot \w{31}
    - (\w{22} \hs \w{31}) \cdot \w{13} 
    - (\w{31} \hs \w{13}) \cdot \w{22}
    - (\w{31} \hs \w{22}) \cdot \w{31}.
  \end{align*}
  A similar replacement now needs to be done for the remaining half-shuffle products.
  As an example of how this works, consider only the terms with $\w{33}$ at the end
  \begin{align*}
    \quad &(\w{11} \hs \w{22}) \cdot \w{33} +
    (\w{22} \hs \w{11}) \cdot \w{33} - 
    (\w{12} \hs \w{21}) \cdot \w{33} -
    (\w{21} \hs \w{12}) \cdot \w{33}  \\
       &= \w{11} \cdot \w{22} \cdot \w{33} + (\w{1} \shuffle \w2) \cdot \w{12} \cdot \w{33}  + \w{22} \cdot \w{11} \cdot \w{33} + (\w{2} \shuffle \w1) \cdot \w{21} \cdot \w{33} \\
      &-\w{12} \cdot \w{21} \cdot \w{33} - (\w{1} \shuffle \w2) \cdot \w{21} \cdot \w{33}  - \w{21} \cdot \w{12} \cdot \w{33} - (\w{2} \shuffle \w{1}) \cdot \w{12} \cdot \w{33}  \\
     &= \w{11} \cdot \w{22} \cdot \w{33} 
 + \w{22} \cdot \w{11} \cdot \w{33} 
 - \w{12} \cdot \w{21} \cdot \w{33} 
 - \w{21} \cdot \w{12} \cdot \w{33},
  \end{align*}
  
  Similarly, applying this procedure for all the other terms, we obtain
 \begin{align*}
     \star = & + \w{11} \cdot \w{33} \cdot \w{22}
      +\w{11} \cdot \w{22} \cdot \w{33}
      +\w{33} \cdot \w{11} \cdot \w{22}
      +\w{22} \cdot \w{11} \cdot \w{33}
      +\w{22} \cdot \w{33} \cdot \w{11}
      +\w{33} \cdot \w{22} \cdot \w{11}\\
     &+\w{31} \cdot \w{23} \cdot \w{12}
      +\w{31} \cdot \w{12} \cdot \w{23}
      +\w{12} \cdot \w{31} \cdot \w{23}
      +\w{12} \cdot \w{23} \cdot \w{31}
      +\w{23} \cdot \w{31} \cdot \w{12}
      +\w{23} \cdot \w{12} \cdot \w{31}\\ 
     &+\w{13} \cdot \w{21} \cdot \w{32}
      +\w{13} \cdot \w{32} \cdot \w{21}
      +\w{21} \cdot \w{13} \cdot \w{32}
      +\w{21} \cdot \w{32} \cdot \w{13}
      +\w{32} \cdot \w{13} \cdot \w{21}
      +\w{32} \cdot \w{21} \cdot \w{13}\\
     &-\w{11} \cdot \w{23} \cdot \w{32}
      -\w{11} \cdot \w{32} \cdot \w{23}
      -\w{23} \cdot \w{11} \cdot \w{32}
      -\w{32} \cdot \w{11} \cdot \w{23}
      -\w{23} \cdot \w{32} \cdot \w{11}
      -\w{32} \cdot \w{23} \cdot \w{11}\\
     &-\w{12} \cdot \w{21} \cdot \w{33}
      -\w{12} \cdot \w{33} \cdot \w{21}
      -\w{33} \cdot \w{12} \cdot \w{21}
      -\w{21} \cdot \w{12} \cdot \w{33}
      -\w{21} \cdot \w{33} \cdot \w{12}
      -\w{33} \cdot \w{21} \cdot \w{12}\\
     &-\w{13} \cdot \w{31} \cdot \w{22}
      -\w{31} \cdot \w{13} \cdot \w{22}
      -\w{13} \cdot \w{22} \cdot \w{31}
      -\w{31} \cdot \w{22} \cdot \w{13}
      -\w{22} \cdot \w{13} \cdot \w{31}
      -\w{22} \cdot \w{31} \cdot \w{13},
  \end{align*}
  as desired.
  
\end{example}

\begin{lemma}\label{lemma:letters}
  For any (non-commutative) polynomial $P$ and for any letters $\w{a}$, $\w{b}$, $\w{c}$, $\w{d}$ we have:
  \begin{align}
      (P \hs [\w{a}\w{b}]) \hs [\w{c}\w{d}]= (P \hs [\w{a}\w{b}])\cdot \w{c}\cdot \w{d} + (P \shuffle \w{a}\shuffle \w{c})\cdot \w{b}\cdot \w{d}.
  \end{align}

\end{lemma}

\begin{proof}
Using (\ref{eq:dendriform}), we have:
\begin{align*}
(P \hs [\w{a}\w{b}]) \hs [\w{c}\w{d}] &= (P \hs [\w{a}\w{b}])\shuffle \w{c})\cdot \w{d} = ((P\shuffle \w{a})\cdot \w{b})\shuffle \w{c})\cdot \w{d}\\
&=((P\shuffle \w{a})\cdot \w{b})\cdot  \w{c})\cdot \w{d} + ((P\shuffle \w{a})\shuffle \w{c})\cdot \w{b}\cdot \w{d}=(P \hs [\w{a}\w{b}])\cdot \w{c}\cdot \w{d}  + (P \shuffle\w{a}\shuffle \w{c})\cdot \w{b}\cdot \w{d}.
\end{align*}
\end{proof}

\begin{proof}[Proof of Proposition~\ref{prop:directProof}]
  The second equality is obtained as follows,
  \begin{multline*}
    \sum_{\tau\in \mathrm{S}_d}\mathrm{sign}(\tau)\bigg (\sum_{\sigma\in \mathrm{S}_d} \aprod_{i=1}^d \sigma(\word{i})\tau(\sigma\big (\word{i})\big )\bigg ) 
    =\sum_{\tau, \sigma \in\mathrm{S}_d} \sign( \tau ) \aprod_{i=1}^d \sigma(\word{i})\tau(\sigma (\word{i}))\\
    =\sum_{\tau, \sigma \in\mathrm{S}_d} \sign( \sigma ) \sign( \tau \circ \sigma ) \aprod_{i=1}^d \sigma(\word{i})\tau\circ\sigma (\word{i})
    =\sum_{\eta, \rho \in\mathrm{S}_d} \sign( \eta ) \sign( \rho )  \aprod_{i=1}^d \eta(\word{i})\rho (\word{i}) =
    \inv\left(
      \begin{gathered}
        \begin{ytableau}
          1 & 2 \\
          3 & 4 \\
          \vdots & \vdots \\
          \scalebox{0.4}{ $\w{2d-1}$ } & \w{2d}
        \end{ytableau}
      \end{gathered}
    \right).
  \end{multline*}

  Regarding the first equality,
  by Lemma~\ref{lemma:andreief}, we have
  \begin{align*}
    {\det}_\shuffle( \Wd )
    &=
    \sum_{\sigma\in \mathrm{S}_d}\sign(\sigma)\niceshuffle_{i=1}^d\ \word{i} \sigma(\word{i}) 
    =\sum_{\sigma, \tau \in \mathrm{S}_d} \sign(\sigma)\sign(\tau) \left[\sigma(\w1)\tau(\w1) \hs \sigma(\w2)\tau(\w2) \hs \dots \hs \sigma(\w{d})\tau(\w{d})\right] \\
    &=\sum_{\sigma, \tau \in \mathrm{S}_d} \sign(\tau)\sign(\sigma \circ \tau) \left[\tau(\w1)\sigma(\tau(\w1)) \hs \tau(\w2)\sigma(\tau(\w2)) \hs \dots \hs \tau(\w{d})\sigma(\tau(\w{d}))\right],
  \end{align*}
  where the last equality holds by rewriting the summation (and reusing the symbols $\sigma$ and $\tau$).
  Note that $\sign(\tau)\sign(\sigma \circ \tau) = \sign(\tau)\sign(\sigma)\sign(\tau)=\sign(\sigma)$, and so $\sign(\tau)$ does not appear in our summations anymore. 

Then, we denote $\dm := \w{d-1}$ and we split our summation into two, depending on the values of $\tau(\dm)$ and $\tau(\w{d})$:
\begin{align*}
& \sum_{\sigma, \tau \in \mathrm{S}_d} \sign(\sigma) \left[\tau(\w1)\sigma(\tau(\w1)) \hs \tau(\w2)\sigma(\tau(\w2)) \hs \dots \hs \tau(\w{d})\sigma(\tau(\w{d}))\right] \\
&= \sum_{\substack{\sigma, \tau \in \mathrm{S}_d \\ \tau(\dm) < \tau(\w{d})}} \sign(\sigma) \left[\tau(\w1)\sigma(\tau(\w1)) \hs \tau(\w2)\sigma(\tau(\w2)) \hs \dots \hs \tau(\w{d})\sigma(\tau(\w{d}))\right] \\
& + \sum_{\substack{\sigma, \tau \in \mathrm{S}_d \\ \tau(\dm) > \tau(\w{d})}} \sign(\sigma) \left[\tau(\w1)\sigma(\tau(\w1)) \hs \tau(\w2)\sigma(\tau(\w2)) \hs \dots \hs \tau(\w{d})\sigma(\tau(\w{d}))\right]
\end{align*}
These summations contain terms that cancel each other, and so we want to control those terms. To do so, we first rewrite the summation for $\tau(\dm) > \tau(\w{d})$ in the following way:
\begin{align*}
&  \sum_{\substack{\sigma, \tau \in \mathrm{S}_d \\ \tau(\dm) > \tau(\w{d})}} \sign(\sigma) \left[\tau(\w1)\sigma(\tau(\w1)) \hs \dots \hs \tau(\w{d})\sigma(\tau(\w{d}))\right] \\
&=\sum_{\substack{\eta,\delta \in \mathrm{S}_d \\ \delta(\dm) < \delta(\w{d})}}
\sign(\eta) \Big[\delta(\w1)\eta(\delta(\w1)) 
\hs \dots 
\hs (\delta(\w{d-2})\eta(\delta(\w{d-2}))\Big.
\hs (\delta (\w{d}))\eta((\delta(\w{d})) \hs
\delta (\dm)\eta(\delta(\dm))\Big],
\end{align*}
where we rewrite the summation by taking $\sigma = \eta$ and $\delta = \tau\circ(d,d-1)$, so that $\delta(\w{i})=\tau(\w{i})$ for $\w{1}\leq \w{i} \leq \w{d-2}$, $\delta(\dm)=\tau(\w{d})$, and $\delta(\w{d}) =\tau (\dm)$. 

Next, we apply Lemma~\ref{lemma:letters} to both summations. For the first summation, we have that 
\begin{align*}
& \sum_{\substack{\sigma, \tau \in \mathrm{S}_d \\ \tau(\dm) < \tau(\w{d})}} \sign(\sigma) \left[\tau(\w1)\sigma(\tau(\w1)) \hs \dots \hs \tau(\w{d})\sigma(\tau(\w{d}))\right] \\ &=
\sum_{\substack{\sigma, \tau \in \mathrm{S}_d \\ \tau(\dm) < \tau(\w{d})}}
\sign(\sigma) \left[\tau(\w1)\sigma(\tau(\w1)) \hs \dots \hs \tau(\dm)\sigma(\tau(\dm))\right]\tau(\w{d})\sigma(\tau(\w{d})) \\
&+ \sum_{\substack{\sigma, \tau \in \mathrm{S}_d \\ \tau(\dm) < \tau(\w{d})}}
\sign(\sigma) \left((\left[\tau(\w1)\sigma(\tau(\w1)) \hs \dots \hs \tau(\w{d-2})\sigma(\tau(\w{d-2}))\right]\shuffle \tau(\dm) \shuffle \tau(\w{d}) \right))\sigma(\tau(\dm))\sigma(\tau(\w{d})), 
\end{align*}
while for the second summation we have that
\begin{align*}
& \sum_{\substack{\eta,\delta \in \mathrm{S}_d \\ \delta(\dm) < \delta(\w{d})}}
\sign(\eta) \Big[\delta(\w1)\eta(\delta(\w1)) 
\hs \dots \hs (\delta(\w{d-2})\eta(\delta(\w{d-2}))\Big. \hs (\delta (\w{d}))\eta((\delta(\w{d}))\Big] \hs
\delta (\dm)\eta(\delta(\dm))\Big] \\ 
&=  \sum_{\substack{\eta,\delta \in \mathrm{S}_d \\ \delta(\dm) < \delta(\w{d})}}
\sign(\eta) \left[\delta(\w1)\eta(\delta(\w1))  \hs \dots \hs \delta(\w{d})\eta(\delta(\w{d}))\right]\delta(\dm)\eta(\delta(\dm)) \\
&+ \sum_{\substack{\eta, \delta \in \mathrm{S}_d \\ \delta(\dm) < \delta(\w{d})}}
\sign(\eta) \left((\left[\delta(\w1)\eta(\delta(\w1)) \hs
\dots  \hs \delta(\w{d-2})\eta(\delta(\w{d-2}))\right]\shuffle \delta(\w{d} \shuffle \delta(\dm) \right))\eta(\delta(\w{d}))\eta(\delta(\dm)). 
\end{align*}

Now we are ready to identify the terms that cancel with each other. In fact, we want to see that
\begin{align}
    &\sum_{\substack{\sigma, \tau \in \mathrm{S}_d \\ \tau(\dm) < \tau(\w{d})}}
\sign(\sigma) \left((\left[\tau(\w1)\sigma(\tau(\w1))  \hs \dots  \hs\tau(\w{d-2})\sigma(\tau(\w{d-2}))\right]\shuffle \tau(\dm) \shuffle \tau(\w{d}) \right))\sigma(\tau(\dm))\sigma(\tau(\w{d})) \nonumber \\
&+ \sum_{\substack{\eta, \delta \in \mathrm{S}_d \\ \delta(\dm) < \delta(\w{d})}}
\sign(\eta) \left((\left[\delta(\w1)\eta(\delta(\w1)) \hs \dots  \hs\delta(\w{d-2})\eta(\delta(\w{d-2}))\right]\shuffle \delta(\w{d} \shuffle \delta(\dm) \right))\eta(\delta(\w{d}))\eta(\delta(\dm))=0. \label{eq:cancellations}
\end{align}
In order to prove~\eqref{eq:cancellations}, we identify terms from each summation. Given $\sigma$, $\tau \in \mathrm S_d$, consider the unique permutations $\eta$, $\delta\in \mathrm S_d$ such that $\delta = \tau$ and $\eta = \sigma\circ\tau\circ(d,d-1)\circ\tau^{-1}$. Thus, $\delta(\w{i})\eta(\delta(\w{i}))$, for $\w{1}\leq\w{i}\leq \w{d-2}$, $\eta(\delta(\w{d}))\eta(\delta(\dm))  = \sigma(\tau(\dm))\sigma(\tau(\w{d})),$ and $\sign(\eta) = -\sign(\sigma)$. This implies that the terms in the first summation of~\eqref{eq:cancellations} cancel with the terms appearing in the second summation. Therefore,
\begin{align*}
& \sum_{\sigma, \tau \in \mathrm{S}_d} \sign(\sigma) \left[\tau(\w1)\sigma(\tau(\w1)) \hs \tau(\w2)\sigma(\tau(\w2)) \hs \dots \hs \tau(\w{d})\sigma(\tau(\w{d}))\right] \\
& = \sum_{\substack{\sigma, \tau \in \mathrm{S}_d \\ \tau(\dm) < \tau(\w{d})}}
\sign(\sigma) \left[\tau(\w1)\sigma(\tau(\w1)) \hs \dots \hs \tau(\dm)\sigma(\tau(\dm))\right]\tau(\w{d})\sigma(\tau(\w{d})) \\
& + \sum_{\substack{\eta,\delta \in \mathrm{S}_d \\ \delta(\dm) < \delta(\w{d})}}
\sign(\eta) \left[\delta(\w1)\eta(\delta(\w1))  \hs \dots \hs \delta(\w{d})\eta(\delta(\w{d}))\right]\delta(\dm)\eta(\delta(\dm)).
\end{align*}
Note that if we take $\eta = \sigma$ and $\delta = \tau\circ(d,d-1)$, the second summation corresponds to the condition $\tau(\dm)> \tau(\w{d})$ and the summations merge into one:
\begin{align*}
& \sum_{\sigma, \tau \in \mathrm{S}_d} \sign(\sigma) \left[\tau(\w1)\sigma(\tau(\w1)) \hs \tau(\w2)\sigma(\tau(\w2)) \hs \dots \hs \tau(\w{d})\sigma(\tau(\w{d}))\right] \\
& = \sum_{\sigma, \tau \in \mathrm{S}_d}
\sign(\sigma) \left[\tau(\w1)\sigma(\tau(\w1)) \hs \dots \hs \tau(\dm)\sigma(\tau(\dm))\right]\tau(\w{d})\sigma(\tau(\w{d}))
\end{align*}
Our final step is basically using the inductive hypothesis on the remaining half-shuffle corresponding to $\dm$. However, one should notice that the factor $\tau(\w1)\sigma(\tau(\w1)) \hs \dots \hs \tau(\dm)\sigma(\tau(\dm))$
does not correspond directly to a term in $\mathrm S_{d-1}$. However, a permutation can be seen as a comparison between total order relations on a set of objects. In this way, for the inductive step, we can consider the restriction of this total order relation to a subset of $d-1$ elements. Thus, replacing step by step all the half-shuffle products by concatenation, we obtain the desired identity 
 \begin{align*}
    \sum_{\sigma, \tau \in \mathrm{S}_d}\sign(\sigma)\niceshuffle_{i=1}^d\ \word{i} \sigma(\word{i})
  & = \sum_{\sigma\in \mathrm{S}_d}\sign(\sigma)\bigg (\sum_{\tau\in \mathrm{S}_d} \aprod_{i=1}^d \tau(\word{i})\sigma(\tau\big (\word{i})\big )\bigg ).
  \end{align*}
\end{proof}

\section{Main result}\label{sec:main}

The aim of this section is to prove the following theorem.
\begin{theorem}
\label{thm:main}
 For $d\geq 1$,
\begin{align*}
    {\det}_\shuffle( \Wd )
    =
    \inv\left(
      \tTwodBig
    \right)
    =
    \frac{1}{2^d}
    \inv\left(
      \tOnedBig
    \right)^{\shuffle 2},
  \end{align*}
  i.e., using the notation \eqref{eq:twoSTYs},
    ${\det}_\shuffle( \Wd ) = \inv(\tTwod) = \frac{1}{2^d} \inv(\tOned)^{\shuffle 2}$.
\end{theorem}
Before presenting the proof for Theorem~\ref{thm:main}, we show that Theorem~\ref{thm:CGM} is a consequence of it. 
\begin{corollary}
  \label{cor:relationToCGM}
  Equation~\eqref{eq:CGM} holds.
\end{corollary}
\begin{proof}

  \begin{align*}
    &\det
    \begin{pmatrix}
      \displaystyle{\int}  dX^{(1)} dX^{(1)} & \dots & \displaystyle{\int}  dX^{(1)} dX^{(d)} \\
             \vdots           & \ddots & \vdots \\
      \displaystyle{\int}  dX^{(d)} dX^{(1)} & \dots & \displaystyle{\int}  dX^{(1)} dX^{(d)}
    \end{pmatrix} 
    =\det
    \begin{pmatrix}
      \langle \w1\w1, \sigma(X) \rangle & \dots & \langle \w1\w{d}, \sigma(X) \rangle \\
             \vdots           & \ddots & \vdots \\
      \langle \w{d}\w1, \sigma(X) \rangle & \dots & \langle \w{d}\w{d}, \sigma(X) \rangle \\[0.6in]
    \end{pmatrix}  \\ &\stackrel{\eqref{eq:shuffleIdentity}}=
    \Big\langle {\det}_\shuffle( \Wd ), \sigma( X ) \Big\rangle 
    \stackrel{\text{Thm }~\ref{thm:main}}=
    \Big\langle \frac{1}{2^d} \inv\left( \tOned \right)^{\shuffle 2}, \sigma( X ) \Big\rangle =
    \frac{1}{2^d} \Big\langle \inv\left( \tOned \right)^{\shuffle 2}, \sigma( X ) \Big\rangle \\[5pt] &\stackrel{\eqref{eq:shuffleIdentity}}=
    \frac{1}{2^d} \Big\langle  \inv\left( \tOned \right), \sigma( X ) \Big\rangle^2  
    = \frac{1}{2^d} \left( \sum_{\sigma \in \mathrm S_d} \sign(\sigma) \int dX^{\sigma(1)} \dots dX^{\sigma(d)} \right)^2.
  \end{align*}
\end{proof}

From Proposition~\ref{prop:directProof} we already know that the first equality holds.
To show the second equality we use the following strategy:
\begin{enumerate}
  \item[\textbf{Step 1}] Show that the two terms lie in the same isotypic component of a representation of a certain subgroup of $\mathrm S_{2d}$, and that this isotypic component has dimension 1.
  \item[\textbf{Step 2}] Determine the pre-factors, by looking at the factors on a particular word in each side of the equality.
\end{enumerate}

For \emph{Step 1}, we first observe
that both $\inv(\tTwod)$
and $\inv(\tOned)^{\shuffle 2}$ are in the irreducible
representation for $\mathrm S_{2d}$ corresponding to the shape
$(2,2,\dots, 2)$. Denote by $\chi_1$ the irreducible character for this shape.

Recall that we denote by $H_d$ the subgroup of $\mathrm S_{2d}$ given by 
\begin{align*}
    H_d = \Big\langle (1,3), (2,4), (3,5), \dots, (2d-2,2d) \Big\rangle \cong \mathrm S_d \times \mathrm S_d.
\end{align*}

Note that $\chi_1$ is also a character for $H_d$, but it is not irreducible.
Let $\chi_2(h) := \operatorname{sign}( h )$, for all $h \in H_d$, be the character for the one-dimensional sign-representation restricted from $\mathrm S_{2d}$ to $H_d$.
This representation is still irreducible, since it is one-dimensional.

Recall the two standard Young tableaux defined in~\eqref{eq:twoSTYs}.
The following two results show that both $(\inv(\tOned)^{\shuffle 2}$ and $\inv(\tTwod)$
lie in the isotypic component related of this sign-representation. 

\begin{lemma}  \label{lem:invSquare}
For all $h\in H_d$, 
  \begin{align*}
    h\cdot \inv_d\left(\tOnedBig\right)^{\shuffle 2} = \operatorname{sign}(h) \cdot \inv_d\left(\tOnedBig\right)^{\shuffle 2}.
  \end{align*}
\end{lemma}

\begin{proof}
Note that it is enough to show this result for generators of the subgroup $H_d$,
in particular permutations of the form $h=(i,i+2)$, with $i\in \{1,2,\dots, 2d-2\}$
are sufficient.

By \eqref{eq:inv1}, we have that
$\displaystyle{\inv(\tOned)^{\shuffle 2} = \sum_{\sigma,\tau\in \mathrm S_d} \sign(\sigma)\sign(\tau) \left[ \sigma(\w{12}\dots\w{d}) \shuffle \tau(\w{12}\dots\w{d})\right]}$, 
and we want to show that 
\begin{multline}
\sum_{\sigma,\tau\in \mathrm S_d} \sign(\sigma)\sign(\tau) \left[ \sigma(\w{12}\dots\w{d}) \shuffle \tau(\w{12}\dots\w{d})\right]\label{eq:expansion} \\ =
\sum_{\sigma,\tau\in \mathrm S_d} \sign((i,i+2))\sign(\sigma)\sign(\tau) (i,i+2) \left[ \sigma(\w{12}\dots\w{d}) \shuffle \tau(\w{12}\dots\w{d})\right]. 
\end{multline}
If we expand $\sigma(\w{12}\dots\w{d}) \shuffle \tau(\w{12}\dots\w{d})$ in each side and do not simplify the expression, we obtain on both sides a huge sum of words with coefficients in $\{\pm 1\}$. To see that both expressions are equal, we show that there is a one-to-one correspondence between the terms in both sums, and that the signs are equal. 

We are considering all terms appearing in the sums in Equation~\eqref{eq:expansion} and we identify each word $\w{w}$ with a triplet $[\sigma,\tau,P]$, where $P$ keeps track of the position of the letters coming from $\sigma(\w{12}\dots\w{d})$ as an increasing list. To visualize our argument, we color-code the letters $\w{12}\dots \w{d}$ when they are letters coming from $\sigma(\w{12}\dots \w{d})$ and $\ww{12}\dots \ww{d}$ when they are letters coming from $\tau(\w{12}\dots \w{d})$. 

Let $\w{w}=\w{w}_1\dots \w{w}_i\w{w}_{i+1} \w{w}_{i+2}\dots \w{w}_{2d}$ be a word appearing in $\sigma(\w{12}\dots\w{d}) \shuffle \tau(\w{12}\dots\w{d})$. This word corresponds to a triplet $[\sigma,\tau,P]$. Now, we apply $h=(i,i+2)$ to $\w{w}$ and we want to describe the triplet $[\sigma',\tau',P']$ corresponding to $h\w{w}$. That is, we want to describe the one-to-one correspondence
\[
\begin{array}{ccc}
\w{w}=\w{w}_1\dots \w{w}_i\w{w}_{i+1} \w{w}_{i+2}\dots \w{w}_{2d} & \longmapsto & h\w{w}=\w{w}_1\dots \w{w}_{i+2}\w{w}_{i+1} \w{w}_{i}\dots \w{w}_{2d} \\[0.15in]
[\sigma,\tau,P] & \longmapsto & [\sigma',\tau',P']
\end{array}
\]
We describe this map by cases depending on whenever the letters $\w{w}_i$, $\w{w}_{i+1}$ and $\w{w}_{i+2}$ are from $\sigma(\w{12}\dots\w{d})$ or from $\tau(\w{12}\dots\w{d})$; that is, depending on the \emph{colors} of the letters $\w{w}_i$, $\w{w}_{i+1}$, and $\w{w}_{i+2}$. Before describing all the possible cases, we need some more notation. Recall that $\sigma$ and $\tau$ act on the left, which means that the act on the positions of $\w{12}\dots\w{d}$. Therefore, $\sigma(\w{12}\dots\w{d})=\sigma^{-1}(\w{1})\sigma^{-1}(\w{2})\dots\sigma^{-1}(\w{d})$ and $\tau(\w{12}\dots\w{d})=\tau^{-1}(\w{1})\tau^{-1}(\w{2})\dots\tau^{-1}(\w{d})$. This implies that the letters appearing in $\w{w}$ corresponds to $\sigma^{-1}(\w{j})$ or $\tau^{-1}(\w{j})$, for some $\w{j}\in [\w{d}]$. 

Now, we are ready to describe all the possible cases. We only look at the relevant part of $\w{w}$, and we include an example using $\sigma = (1,3,4)$, $\tau = (1,2,3)$ and $\w{4213} \shuffle  \ww{3124}$ to illustrate each case. 
\begin{enumerate}
    \item[(C1)] If the three letters are coming from $\sigma(\w{12}\dots\w{d})$, then
\[
\begin{array}{ccc||ccc}
 & \w{w} & &  & h\w{w} & \\[0.1in]
  \multicolumn{1}{c|}{\w{w}_i} &  \multicolumn{1}{c|}{\w{w}_{i+1}} & \w{w}_{i+2}  &  \multicolumn{1}{c|}{\w{w}_{i+2}} &  \multicolumn{1}{c|}{\w{w}_{i+1}} & \w{w}_{i} \\[0.1in]
  \multicolumn{1}{c|}{\sigma^{-1}(\w{j})} &  \multicolumn{1}{c|}{\sigma^{-1}(\w{j+1})} & \sigma^{-1}(\w{j+2}) &   \multicolumn{1}{c|}{\sigma^{-1}(\w{j+2})} &  \multicolumn{1}{c|}{\sigma^{-1}(\w{j+1})} & \sigma^{-1}(\w{j})
\end{array}
\]
Therefore, $\sigma'=(j,j+2)\sigma$, $\tau'=\tau$, and $P'=P$ since the position of the blue letters does not change.

    \item[(C2)] If two of the letters are coming from $\sigma(\w{12}\dots\w{d})$ and the other one is coming from $\tau(\w{12}\dots\w{d})$, then we have three cases:
    \begin{itemize}
        \item[(C2.a)]
\[
\begin{array}{ccc||ccc}
 & \w{w} & & & h\w{w} & \\[0.1in]
  \multicolumn{1}{c|}{\w{w}_i} &  \multicolumn{1}{c|}{\w{w}_{i+1}} & \ww{w}_{i+2}  &  \multicolumn{1}{c|}{\ww{w}_{i+2}} &  \multicolumn{1}{c|}{\w{w}_{i+1}} & \w{w}_{i} \\[0.1in]
 \multicolumn{1}{c|}{\sigma^{-1}(\w{j})} &  \multicolumn{1}{c|}{\sigma^{-1}(\w{j+1})} & \tau^{-1}(\w{k}) &   \multicolumn{1}{c|}{\tau^{-1}(\w{k})} &  \multicolumn{1}{c|}{\sigma^{-1}(\w{j+1})} & \sigma^{-1}(\w{j})
\end{array}
\]
Therefore, $\sigma'=(j,j+1)\sigma$ and $\tau'=\tau$. Moreover, we know that $P$ is of the form $P=(\dots, i, i+1,\dots)$, with no $i+2$ in $P$. Therefore, $P'$ is of the form $P'=(\dots, i+1,i+2,\dots)$, with no $i$ in $P'$. 

        \item[(C2.b)]
\[
\begin{array}{ccc||ccc}
 & \w{w} & & & h\w{w} & \\[0.1in]
  \multicolumn{1}{c|}{\w{w}_i} &  \multicolumn{1}{c|}{\ww{w}_{i+1}} & \w{w}_{i+2}  &  \multicolumn{1}{c|}{\w{w}_{i+2}} &  \multicolumn{1}{c|}{\ww{w}_{i+1}} & \w{w}_{i} \\[0.1in]
  \multicolumn{1}{c|}{\sigma^{-1}(\w{j})} &  \multicolumn{1}{c|}{\tau^{-1}(\w{k})} & \sigma^{-1}(\w{j+1}) &   \multicolumn{1}{c|}{\sigma^{-1}(\w{j+1})} &  \multicolumn{1}{c|}{\tau^{-1}(\w{k})} & \sigma^{-1}(\w{j})
\end{array}
\]
Therefore, $\sigma'=(j,j+1)\sigma$, $\tau'=\tau$, and $P'=P$ since the position of the blue letters does not change. 

        \item[(C2.c)]
\[
\begin{array}{ccc||ccc}
 & \w{w} & & & h\w{w} & \\[0.1in]
  \multicolumn{1}{c|}{\ww{w}_i} &  \multicolumn{1}{c|}{\w{w}_{i+1}} & \w{w}_{i+2}  &  \multicolumn{1}{c|}{\w{w}_{i+2}} &  \multicolumn{1}{c|}{\w{w}_{i+1}} & \ww{w}_{i} \\[0.1in]
  \multicolumn{1}{c|}{\tau^{-1}(\w{k})} &  \multicolumn{1}{c|}{\sigma^{-1}(\w{j})} & \sigma^{-1}(\w{j+1}) &   \multicolumn{1}{c|}{\sigma^{-1}(\w{j+1})} &  \multicolumn{1}{c|}{\sigma^{-1}(\w{j})} & \tau^{-1}(\w{k})
\end{array}
\]
Therefore, $\sigma'=(j,j+1)\sigma$ and $\tau'=\tau$. Moreover, we know that $P$ is of the form $P=(\dots, i+1, i+2,\dots)$, with no $i$ in $P$. Therefore, $P'$ is of the form $P'=(\dots, i,i+1,\dots)$, with no $i+2$ in $P'$. 
    \end{itemize}

    \item[(C3)] If one of the letters is coming from $\sigma(\w{12}\dots\w{d})$ and the other two are coming from $\tau(\w{12}\dots\w{d})$, then we have three cases:
    \begin{itemize}
        \item[(C3.a)]
\[
\begin{array}{ccc||ccc}
 & \w{w} & &  & h\w{w} & \\[0.1in]
  \multicolumn{1}{c|}{\w{w}_i} &  \multicolumn{1}{c|}{\ww{w}_{i+1}} & \ww{w}_{i+2}  &  \multicolumn{1}{c|}{\ww{w}_{i+2}} &  \multicolumn{1}{c|}{\ww{w}_{i+1}} & \w{w}_{i} \\[0.1in]
  \multicolumn{1}{c|}{\sigma^{-1}(\w{j})} &  \multicolumn{1}{c|}{\tau^{-1}(\w{k})} & \tau^{-1}(\w{k+1}) &   \multicolumn{1}{c|}{\tau^{-1}(\w{k+1})} &  \multicolumn{1}{c|}{\tau^{-1}(\w{k})} & \sigma^{-1}(\w{j})
\end{array}
\]
Therefore, $\sigma'=\sigma$ and $\tau'=(k,k+1)\tau$. Moreover, we know that $P$ is of the form $P=(\dots, i,\dots)$, with no $i+1, i+2$ in $P$. Therefore, $P'$ is of the form $P'=(\dots, i+2,\dots)$, with no $i,i+1$ in $P'$. 

        \item[(C3.b)]
\[
\begin{array}{ccc||ccc}
 & \w{w} & &  & h\w{w} & \\[0.1in]
 \multicolumn{1}{c|}{ \ww{w}_i} &  \multicolumn{1}{c|}{\w{w}_{i+1}} & \ww{w}_{i+2}  &  \multicolumn{1}{c|}{\ww{w}_{i+2}} &  \multicolumn{1}{c|}{\w{w}_{i+1}} & \ww{w}_{i} \\[0.1in]
  \multicolumn{1}{c|}{\tau^{-1}(\w{k})} &  \multicolumn{1}{c|}{\sigma^{-1}(\w{j})} & \tau^{-1}(\w{k+1}) &   \multicolumn{1}{c|}{\tau^{-1}(\w{k+1})} &  \multicolumn{1}{c|}{\sigma^{-1}(\w{j})} & \tau^{-1}(\w{k})
\end{array}
\]
Therefore, $\sigma'=\sigma$, $\tau'=(k,k+1)\tau$, and $P'=P$ since the position of the blue letters does not change.  

        \item[(C3.c)]
\[
\begin{array}{ccc||ccc}
 & \w{w} & &  & h\w{w} & \\[0.1in]
 \multicolumn{1}{c|}{\ww{w}_i} & \multicolumn{1}{c|}{\ww{w}_{i+1}} & \w{w}_{i+2}  & \multicolumn{1}{c|}{\w{w}_{i+2}} & \multicolumn{1}{c|}{\ww{w}_{i+1}} & \ww{w}_{i} \\[0.1in]
 \multicolumn{1}{c|}{\tau^{-1}(\w{k})} & \multicolumn{1}{c|}{\tau^{-1}(\w{k+1})} & \sigma^{-1}(\w{j}) &  \multicolumn{1}{c|}{\sigma^{-1}(\w{j})} & \multicolumn{1}{c|}{\tau^{-1}(\w{k+1})} & \tau^{-1}(\w{k})
\end{array}
\]
Therefore, $\sigma'=\sigma$ and $\tau'=(k,k+1)\tau$. Moreover, we know that $P$ is of the form $P=(\dots, i+2,\dots)$, with no $i, i+1$ in $P$. Therefore, $P'$ is of the form $P'=(\dots, i,\dots)$, with no $i+1,i+2$ in $P'$. 
    \end{itemize}
  
  \item[(C4)] If the three letters are coming from $\tau(\w{12}\dots\w{d})$, we only have one case: 
  \[
\begin{array}{ccc||ccc}
 & \w{w} &  & & h\w{w} & \\[0.1in]
 \multicolumn{1}{c|}{\ww{w}_i} &  \multicolumn{1}{c|}{\ww{w}_{i+1}} & \ww{w}_{i+2}  &  \multicolumn{1}{c|}{\ww{w}_{i+2}} &  \multicolumn{1}{c|}{\ww{w}_{i+1}} & \ww{w}_{i} \\[0.1in]
 \multicolumn{1}{c|}{\tau^{-1}(\w{k})} &  \multicolumn{1}{c|}{\tau^{-1}(\w{k+1})} & \tau^{-1}(\w{k+2})  &  \multicolumn{1}{c|}{\tau^{-1}(\w{k+2})} &  \multicolumn{1}{c|}{\tau^{-1}(\w{k+1}) }& \tau^{-1}(\w{k})
\end{array}
\]
Therefore, $\sigma'=\sigma$, $\tau'=(k,k+2)\tau$, and $P'=P$ since the position of the blue letters does not change. 
\end{enumerate}

Finally, we look at the sign. Notice that in all the cases, $\sign(\sigma)\sign(\tau) = \sign(i,i+2)\sign(\sigma')\sign(\tau')$
because either $\sigma'$ nor $\tau'$ differ from $\sigma$ and $\tau$, respectively, by a transposition. 
\end{proof}

\begin{example}
Consider $\sigma = (1,3,4)$ and $\tau = (1,2,3)$, for which $\sigma(\w{1234}) = \w{4213}$ and $\tau(\w{1234}) = \ww{3124}$. Then, $\w{w}=\w{4}\ww{31}\w{213}\ww{24}$ is a word appearing in $\w{4213} \shuffle  \ww{3124}$. Let us illustrate the cases in the previous proof with this example with $i=4$:
\begin{equation*}
\begin{array}{l||l}
\emph{Case (C1)}  & \emph{Case (C2.a)}   \\ \hline
  \begin{array}{c|c}
  \w{4}\ww{31}\w{213}\ww{24} &  \w{4}\ww{31}\w{312}\ww{24}  \\[0.1in]
  \sigma(\w{1234}) = \w{4213} & \sigma'(\w{1234}) = \w{4312} = (2,4)\sigma \\[0.075in]
  \tau (\w{1234}) = \w{3124} & \tau'(\w{1234}) = \w{3124} \\[0.075in]
  P = [1,4,5,6] & P'=[1,4,5,6] 
\end{array} &  
\begin{array}{c|c}
  \w{42}\ww{3}\w{13}\ww{124} &  \w{42}\ww{31}\w{31}\ww{24}  \\[0.1in]
  \sigma(\w{1234}) = \w{4213} & \sigma'(\w{1234}) = \w{4231} = (3,4)\sigma \\[0.075in]
  \tau (\w{1234}) = \w{3124} & \tau'(\w{1234}) = \w{3124} \\[0.075in]
  P = [1,2,4,5] & P'=[1,2,5,6] 
\end{array} 
\end{array}
\end{equation*}
\begin{equation*}
\begin{array}{l||l}
\emph{Case (C2.b)}  & \emph{Case (C2.c)}   \\ \hline
\begin{array}{c|c}
  \w{42}\ww{3}\w{1}\ww{1}\w{3}\ww{24} &  \w{42}\ww{3}\w{3}\ww{1}\w{1}\ww{24}  \\[0.1in]
  \sigma(\w{1234}) = \w{4213} & \sigma'(\w{1234}) = \w{4231} = (3,4)\sigma \\[0.075in]
  \tau (\w{1234}) = \w{3124} & \tau'(\w{1234}) = \w{3124} \\[0.075in]
  P = [1,2,4,6] & P'=[1,2,4,6] 
\end{array} & 
\begin{array}{c|c}
  \w{42}\ww{31}\w{13}\ww{24} &  \w{42}\ww{3}\w{31}\ww{124}  \\[0.1in]
  \sigma(\w{1234}) = \w{4213} & \sigma'(\w{1234}) = \w{4231} = (3,4)\sigma \\[0.075in]
  \tau (\w{1234}) = \w{3124} & \tau'(\w{1234}) = \w{3124} \\[0.075in]
  P = [1,2,5,6] & P'=[1,2,4,5] 
\end{array} 
\end{array}
\end{equation*}
\begin{equation*}
\begin{array}{l||l}
\emph{Case (C3.a)}  & \emph{Case (C3.b) }  \\ \hline
\begin{array}{c|c}
  \w{42}\ww{3}\w{1}\ww{12}\w{3}\ww{4} &  \w{42}\ww{321}\w{13}\ww{4}  \\[0.1in]
  \sigma(\w{1234}) = \w{4213} & \sigma'(\w{1234}) = \w{4213} \\[0.075in]
  \tau (\w{1234}) = \w{3124} & \tau'(\w{1234}) = \w{3214} = (2,3)\tau \\[0.075in]
  P = [1,2,4,7] & P'=[1,2,6,7] 
\end{array} &
\begin{array}{c|c}
  \w{42}\ww{31}\w{1}\ww{2}\w{3}\ww{4} &  \w{42}\ww{32}\w{1}\ww{1}\w{3}\ww{4}  \\[0.1in]
  \sigma(\w{1234}) = \w{4213} & \sigma'(\w{1234}) = \w{4213} \\[0.075in]
  \tau (\w{1234}) = \w{3124} & \tau'(\w{1234}) = \w{3214} = (2,3)\tau \\[0.075in]
  P = [1,2,3,4] & P'=[1,2,3,6] 
\end{array} 
\end{array}
\end{equation*}
\begin{equation*}
\begin{array}{l||l}
\emph{Case (C3.c)}  & \emph{Case (C4)}   \\ \hline
\begin{array}{c|c}
  \w{42}\ww{312}\w{13}\ww{4} &  \w{42}\ww{3}\w{1}\ww{21}\w{3}\ww{4}  \\[0.1in]
  \sigma(\w{1234}) = \w{4213} & \sigma'(\w{1234}) = \w{4213} \\[0.075in]
  \tau (\w{1234}) = \w{3124} & \tau'(\w{1234}) = \w{3214} = (2,3)\tau \\[0.075in]
  P = [1,2,6,4] & P'=[1,2,4,6] 
\end{array} &
\begin{array}{c|c}
  \w{421}\ww{312}\w{3}\ww{4} &  \w{421}\ww{213}\w{3}\ww{4}  \\[0.1in]
  \sigma(\w{1234}) = \w{4213} & \sigma'(\w{1234}) = \w{4213} \\[0.075in]
  \tau (\w{1234}) = \w{3124} & \tau'(\w{1234}) = \w{2134} = (1,3)\tau \\[0.075in]
  P = [1,2,3,7] & P'=[1,2,3,7] 
\end{array}
    \end{array}
\end{equation*}
\end{example}

\vspace{0.3cm}

It is immediate that $\inv(\tTwod)$ satisfies the analogous statement.
\begin{lemma}\label{lem:specialYT}
For all $h\in H_d$, 
  \begin{align*}
    h\cdot
    \inv\left(\tTwodBig\right)
    =
    \operatorname{sign}(h)\cdot
    \inv\left(\tTwodBig\right).
  \end{align*}
\end{lemma}

Now, we want to see that the dimension of this isotypic component is 1. This follows from the fact that the multiplicity of the irreducible character $\chi_2$  in $\chi_1$ is equal to one in the case we are considering.
\begin{lemma} \label{lem:oneDimensional} 
Consider the decomposition (with respect to $H_d$) of the   representation (with respect to $\mathrm S_{2d}$) corresponding to $\chi_1$.
In this decomposition, the multiplicity of the irreducible sign-representation, i.e. the irreducible representation corresponding to $\chi_2$, has multiplicity one.
\end{lemma}
\begin{proof}

We want to look at the \emph{sign representation} of $\mathrm S_{2d}$, $\sign$, in terms of representations of $\mathrm S_d$ since we have that $H_d \cong \mathrm S_d \times \mathrm S_d$. 

Consider $\lambda=\mu=(\underbrace{1,...,1}_{d \text{ times }})$, so that $V_\lambda = V_\mu$ is the sign representation of $\mathrm S_d$, and $\nu = (\underbrace{2,...,2}_{d \text{ times }})$. 

 Denote by $V_\lambda \boxtimes V_\mu$ the (external) product of representations (see
\cite[Exercise 2.36]{bib:FuHa}). Then, we claim that $V_\lambda \boxtimes V_\mu
   \cong
   \sign|_{H_d}$, where $\sign|_{H_d}$ denotes the restriction of the representation $\sign$ to $H_d$. By \cite[Exercise 2.36]{bib:FuHa} all the irreducible representations of $\mathrm  S_d\times\mathrm  S_d$ are given by the external product of irreducible representations of $\mathrm S_d$. In our case, there are exactly two one-dimensional irreducible representations of $\mathrm S_d$, the trivial and the sign representations. Thus, 
there are exactly $4$ different one-dimensional
representations of $\mathrm S_d \times \mathrm S_d$, obtained by taking the external product of either the trivial or the sign representation with, again, either the trivial or the sign representation.
Now, $\sign|_{H_d}$ is one-dimensional, and so it has to be one of these 4 different representations. Therefore, 
$\sign|_{H_d}$ has to correspond to the one claimed, $V_\lambda \boxtimes V_\mu$.

Hence, the multiplicity (as a $H_d$ representation) of $\chi_2$  inside of $\chi_1$ (as $\mathrm S_{2d}$ representation) is given by the multiplicity (as $\mathrm S_d\times \mathrm S_d$ representation)
of $V_\lambda \boxtimes V_\mu$ inside of $V_\nu$.
By
\cite[Exercise 4.43, Appendix A.1]{bib:FuHa}
the latter is given by the Littlewood-Richardson coefficient $c_{\lambda\mu}^{\nu}$. That is, we are counting the number of semi-standard Young tableaux of skew-shape $\nu/\lambda$ and type $\mu$ such that the reading word is a lattice word (see~\cite{bib:Stanley}). In our case, $\nu/\lambda$ is exactly a column and type $\mu$ means that we have to fill out the column exactly with the numbers $1,2,\dots, d$ in increasing order. Since there is only one way to do it and the lattice word condition is trivially satisfied, $c_{\lambda\mu}^\nu =1$. 

\end{proof}

We are now ready to finish the proof of our main theorem. 
\begin{proof}[Proof of Theorem~\ref{thm:main}]
  Let us start by summarising the situation.
  By Proposition~\ref{prop:directProof} we have that ${\det}_\shuffle(\Wd) = \inv(\tTwod)$.
  Now both $\inv(\tTwod)$ and $(\inv(\tOned))^{\shuffle 2}$ live in the irreducible $\mathrm S_{2d}$-representation corresponding to the shape $(2,2,\dots,2)$.
  Restricting this representation to the subgroup $H_d$ of $\mathrm S_{2d}$ and decomposing into irreducible representations of $H_d$,
  we have that $\det_{\shuffle}(\Wd)$ and $(\inv(\tOned))^{\shuffle 2}$ lie in the
  component corresponding to the sign-representation, see Lemmas
 ~\ref{lem:invSquare} and~\ref{lem:specialYT}.
  We also know, by Lemma~\ref{lem:oneDimensional}, that this component is $1$-dimensional.

  Our last step is to prove that the prefactors are as stated.
  Now, the word $\word{1122\dots dd}$ can only be obtained from
  shuffling $\w{12\dots d}$ with $\w{12\dots d}$,
  which results in a factor of $2^d$.
  Using the fact that it reduces to shuffle products of blocks of $2$ letters (Proposition~\ref{prop:directProof}),
  one sees that
  in $\det_{\shuffle}(\Wd)$, the word $\word{1122\dots dd}$ appears with the factor one.
\end{proof}

The following example illustrates the computation of the prefactors from Theorem~\ref{thm:main} for $d=3$. 
\begin{example}
  The word $\word{112233}$
  appearing in
  $\inv(\mathfrak t_{1,3})^{\shuffle 2}$,
  can only come from the shuffle $\w{123} \shuffle \w{123}$.
  This can happen in $2^3 = 8$ ways, 
  namely (using the color-code from the proof of Lemma~\ref{lem:invSquare})
    \begin{center}
  $\w{1}\ww{1}\w{2}\ww{2}\w{3}\ww{3}$, $\w{1}\ww{12}\w{23}\ww{3}$, 
  $\w{1}\ww{1}\w{2}\ww{23}\w{3}$, $\w{1}\ww{12}\w{2}\ww{3}\w{3}$, 
  $\ww{1}\w{12}\ww{2}\w{3}\ww{3}$, $\ww{1}\w{1}\ww{2}\w{23}\ww{3}$,
  $\ww{1}\w{12}\ww{23}\w{3}$, $\ww{1}\w{1}\ww{2}\w{2}\ww{3}\w{3}$.
  \end{center}
  
  On the other hand, in $\det_{\shuffle}(\mathcal{W}_3)$ it only appears once in the term coming from $\w{11} \shuffle \w{22} \shuffle \w{33}$
  (the diagonal of the matrix in the Leibniz formula for the determinant).
\end{example}

\section{Pfaffians and de Bruijn's formula}
\label{subsec:deBruijn}

Before stating and proving de Bruijn's formula,
we recall some definitions and present some technical results. 
Let $\mathcal A$ be a commutative algebra over $\R$.

\begin{lemma}\label{lem:detAS}
  Let $A \in \mathcal A^{d\times d}$ anti-symmetric,
  $v \in \mathcal A^d$ and $\lambda \in \mathcal A$. 
  \begin{itemize}
  \item If $d$ is even, $\det[ A + \lambda v v^\top ] = \det[ A ]$.
  \item If $d$ is odd, 
  $\det[ A + \lambda v v^\top ] = \sum_{i=1}^d \det[ R_i ]$,
  where $R_i$ denotes the matrix $A$, with $i^{\text{th}}$ column by that row of $\lambda v v^\top$.
  \end{itemize}
\end{lemma}
\begin{proof}
  Let $V:=\lambda v v^\top$.
  For $I\subset [n] := \{1,\dots,n\}$,
  let $R_I$ be the matrix $A$ with the rows corresponding to $I$ replaced by the corresponding rows of $V$.
  Correspondingly, let $C_I$ be the matrix $A$ with the columns corresponding to $I$ replaced by the corresponding columns of $V$.
  Then
  \begin{align*}
    \det[ A + V ] &= \sum_{I \subset [n]} \det( R_I ) = \det( A ) + \sum_{i=1}^n \det( R_{\{i\}} ) \\
    \det[ A + V ] &= \sum_{I \subset [n]} \det( C_I ) = \det( A ) + \sum_{i=1}^n \det( C_{\{i\}} ),
  \end{align*}
  where we used the fact that $\det(R_I)=\det(C_I)=0$ if $|I|\ge 2$.
  
  If $d$ is odd, $\det[A] = 0$, since $A$ is anti-symmetric.
  This yields the statement in this case, with ($R_i := R_{\{i\}}$).
  
  If $d$ is even, we add the two expansions to get
  $$\det[ A + V ]
    = \det[ A ] + \frac{1}{2} \sum_{i=1}^n \left( \det( R_{\{i\}} ) + \det( C_{\{i\}} ) \right),$$
    and we get $\det( R_{\{i\}} ) = - \det( C_{\{i\}} )$, for all $i \in [n]$.
  This finishes the proof.
\end{proof}

For the following statement recall the matrix $\Wd$ from~\eqref{eq:Wd} with entries in the shuffle algebra $T(\R^d)$.
\begin{lemma}\label{lem:symmatrix}
 Write $\operatorname{Sym}[\Wd]4:= \frac{1}{2} (\Wd + \Wd^\top)$.
    Then we have
    \begin{align*} \operatorname{Sym}[\Wd]
    &=
    \frac{1}{2}\begin{pmatrix}
      \w1 \\
      \w2 \\
      \vdots \\
      \w{d}
    \end{pmatrix}\shuffle
    \begin{pmatrix}
      \w1 \\
      \w2 \\
      \vdots \\
      \w{d}
    \end{pmatrix}^\top.
  \end{align*}
\end{lemma}

\begin{proof}
  We have the following computation
  \begin{align*}
  \begin{pmatrix}
      \w1 \\
      \w2 \\
      \vdots \\
      \w{d}
    \end{pmatrix}\shuffle
    \begin{pmatrix}
      \w1 \\
      \w2 \\
      \vdots \\
      \w{d}
    \end{pmatrix}^\top
    &=
    \begin{pmatrix}
    \w1\shuffle \w1 & \w1\shuffle \w2 & \ldots & \w1\shuffle \w{d}\\
    \vdots &\vdots &\ddots &\vdots  \\
    \w{d}\shuffle \w1 & \w{d}\shuffle \w2 & \ldots &\w{d}\shuffle \w{d}\\
          \end{pmatrix}
          =
           \begin{pmatrix}
    \w1\w1+\w1\w1 & \w1\w2+\w2\w1 & \ldots & \w1\w{d}+\w{d}\w1\\
         \vdots &\vdots &\ddots &\vdots  \\
    \w{d}\w1+\w1\w{d} & \w{d}\w2+\w2\w{d} & \ldots &\w{d}\w{d}+\w{d}\w{d}\\
          \end{pmatrix} =\Wd + \Wd^\top.
      \end{align*}
      \end{proof}
The following definition is well-known (\cite[Chapter 5.2, p.82]{Bourbaki})
\begin{definition}
  Let $A=(a_{ij}) \in \mathcal A^{d\times d}$ be a anti-symmetric matrix, where $d$ is even.
  Then 
  \begin{align*}
    \Pf_{\mathcal A}[A]
    :=
    \frac{1}{2^{d/2}}\frac{1}{(d/2)!}\sum_{\pi\in\mathrm{S}_d}\mathrm{sign}(\pi)\prod_{i=1}^{d/2}a_{\pi(2i-1) \pi(2i)}
  \end{align*}
  is called the \emph{Pfaffian} of $A$.
\end{definition}

For the following statement see for instance~\cite[pages 140-141]{MR0082463},~\cite{MR1221053}, and references therein. 
\begin{lemma}
  \label{lem:pfaffian}
  Let $d$ be even and $A \in \mathcal A^{d\times d}$ a anti-symmetric matrix.
  We have the following identity:
  \begin{align*}
    {\det}_{\mathcal A}[ A ] = \left( \Pf_{\mathcal A}[A] \right)^2.
  \end{align*}  
\end{lemma}

\subsection{Deducing de Bruijn's formula from our results}
  The de Bruijn's formula has two different formulations depending on the parity of $d$. We compare both formulations with our results. 

  We recall that any matrix $M$ can be written as the sum of its symmetric part and its anti-symmetric part
$M = \operatorname{Sym}[M] + \operatorname{Anti}[M]$,
$\operatorname{Sym}[M] := \tfrac12 (M + M^\top)$ and  $\operatorname{Anti}[M] := \tfrac12 (M - M^\top)$.

  For $d$ even, de Bruijn's formula (\cite{bib:deBruijn}, see also~\cite[Remark 3.18]{bib:DR2018}),
  reads, in modern language, as follows.
  \begin{theorem*}[de Bruijn's formula - Even case]
    \label{thm:deBruijnEven}
    \begin{align*}
      \inv\left(
          \begin{gathered}
          \begin{ytableau}
          1  \\
          2 \\
          \vdots\\
           {d} 
          \end{ytableau}
          \end{gathered}
      \right)
      =
      2^{d/2}
      \Pf_\shuffle[ \operatorname{Anti}[\Wd ] ].
    \end{align*}
  \end{theorem*}

  \begin{proof}
  We deduce this from our results in the main text.
  First, Theorem~\ref{thm:main} gives
  \begin{align*}
    \inv\left(
        \begin{gathered}
        \begin{ytableau}
        1  \\
        2 \\
        \vdots\\
         {d} 
        \end{ytableau}
        \end{gathered}
    \right)^{\shuffle 2}
    =
    2^d {\det}_\shuffle[ \Wd ].
  \end{align*}

  Note that $\Wd = \operatorname{Sym}(\Wd) + \operatorname{Anti}(\Wd)$. Then, by Lemma~\ref{lem:symmatrix},
  ${\det}_\shuffle[ \Wd ]=
    {\det}_\shuffle[ \operatorname{Anti}[\Wd] ]$.
  By properties of the Pfaffian, Lemma~\ref{lem:pfaffian},
  ${\det}_\shuffle[ \operatorname{Anti}[\Wd] ] =
    {\Pf}_\shuffle[ \operatorname{Anti}[\Wd] ]^{\shuffle 2}$.
  Hence
  \begin{align*}
    \inv\left(
        \begin{gathered}
        \begin{ytableau}
        1  \\
        2 \\
        \vdots\\
         {d} 
        \end{ytableau}
        \end{gathered}
    \right)^{\shuffle 2}
    = 
   2^d {\Pf}_\shuffle[ \operatorname{Anti}[\Wd] ]^{\shuffle 2}.
  \end{align*}
  Since the shuffle algebra is commutative and integral, we deduce that: 
  \begin{align*}
    \inv\left(
        \begin{gathered}
        \begin{ytableau}
        1  \\
        2 \\
        \vdots\\
         {d} 
        \end{ytableau}
        \end{gathered}
    \right)
    =
    \pm 2^{d/2} {\Pf}_\shuffle[ \operatorname{Anti}[\Wd] ].
  \end{align*}

  By comparing the coefficient of the word $\w{123\dots d}$, we deduce that the sign must be positive, and
  have thus shown de Bruijn's formula in the even case. 
  \end{proof}
  
  \begin{example}
de Bruijn's formula in the case $d=4$ gives
\begin{align*}
\inv\left(
        \begin{gathered}
        \begin{ytableau}
        1  \\
        2 \\
        {3} \\
        {4}
        \end{ytableau}
        \end{gathered}
    \right)
  =
  4\operatorname{Pf}_\shuffle
  \begin{pmatrix}
    0 & \w{12-21} & \w{13-31}         & \w{14-41} \\
    \w{21-12} & 0 & \w{23-32}         & \w{24-42} \\
    \w{31-13} & \w{32-23} & 0         & \w{34-43} \\
    \w{41-14} & \w{42-24} & \w{43-34} &         0
  \end{pmatrix},
\end{align*}
whereas Theorem~\ref{thm:main} gives
\begin{align*}
  \inv\left(
        \begin{gathered}
        \begin{ytableau}
        1  \\
        2 \\
        {3} \\
        {4}
        \end{ytableau}
        \end{gathered}
     \right)^{\shuffle 2}
  =
  16\
  {\det}_\shuffle 
  \begin{pmatrix}
    \w{11} & \w{12} & \w{13} & \w{14}\\
    \w{21} & \w{22} & \w{23} & \w{24}\\
    \w{31} & \w{32} & \w{33} & \w{34} \\
    \w{41} & \w{42} & \w{43} & \w{44}
  \end{pmatrix}.
\end{align*}

\end{example}

Now, we look at the odd case. 
  \begin{theorem*}[de Bruijn's formula - Odd case]
    \label{thm:deBruijnOdd}
For $d$ odd, de Bruijn's formula (\cite{bib:deBruijn}) reads as
  \begin{align*}
    \inv\left(
        \begin{gathered}
        \begin{ytableau}
        1  \\
        2 \\
        \vdots\\
        {d} 
        \end{ytableau}
        \end{gathered}
    \right)
    =
    \Pf_\shuffle[ Z_d ],
  \end{align*}
  where $Z_d$ denotes the matrix of the form
  \begin{align*}
    Z_d
    =
    \left( \begin{array}{c|c} 2\operatorname{Anti}[\Wd] & \begin{array}{c} \w1 \\ \w2 \\ \vdots \\ \w{d} \end{array} \\ \hline \begin{array}{cccc} - \w1 & -\w2 & \cdots & -\w{d} \end{array} & 0 \end{array} \right).
  \end{align*}
  \end{theorem*}

  \begin{proof}
We follow a similar strategy as in the even-dimensional case,
to show this theorem using our results.
As before, we decompose into symmetric and anti-symmetric part,
  $\Wd = \operatorname{Sym}[\Wd] + \operatorname{Anti}[\Wd]$.
By Lemma~\ref{lem:symmatrix}, $\operatorname{Sym}[\Wd]= \frac{1}{2}\w{v}\w{v^\top}$ for $\w{v}^\top = \left[\begin{array}{cccc} \w1 & \w2 & \cdots & \w{d} \end{array}\right]$.
Then by Lemma~\ref{lem:detAS},
${\det}_\shuffle[\Wd]  =
  \sum_{i=1}^d {\det}_\shuffle\left[ R_i \right]$,
where $R_i$ is the matrix $\operatorname{Anti}[\Wd]$ with $i^{\text{th}}$ row replaced by $\frac12 v_i v^\top$.

Now
\begin{align*}
  \sum_{i=1}^d {\det}_\shuffle\left[ R_i \right]
  =
  2^{-d}
  {\det}_\shuffle\left( \begin{array}{c|c} 2\operatorname{Anti}[\Wd] & \begin{array}{c} \w1 \\ \w2 \\ \vdots \\ \w{d} \end{array} \\ \hline \begin{array}{cccc} - \w1 & -\w2 & \cdots & -\w{d} \end{array} & 0 \end{array} \right)
  =
  2^{-d}  {\det}_\shuffle (Z_d),
\end{align*}
which gives  $2^d {\det}_\shuffle[\Wd] = {\det}_\shuffle (Z_d)$.
Combining with our Theorem~\ref{thm:main}, we get
  \begin{align*}
    \inv\left(
        \begin{gathered}
        \begin{ytableau}
        1  \\
        2 \\
        \vdots\\
        {d} 
        \end{ytableau}
        \end{gathered}
    \right)^{\shuffle 2}
    =
    2^d {\det}_\shuffle[ \Wd ] = {\det}_\shuffle (Z_d).
  \end{align*}
  By Lemma~\ref{lem:pfaffian} (note that $Z_d$ is even-dimensional), 
${\det}_\shuffle[ Z_d] =
    {\Pf}_\shuffle[ Z_d ]^{\shuffle 2}$.
  Hence
  \begin{align*}
    \inv\left(
        \begin{gathered}
        \begin{ytableau}
        1  \\
        2 \\
        \vdots\\
        {d} 
        \end{ytableau}
        \end{gathered}
    \right)^{\shuffle 2}
    = {\Pf}_\shuffle[ Z_d  ]^{\shuffle 2}.
  \end{align*}
  As for the even case, since the shuffle algebra is commutative and integral, and by comparing the coefficient of the word $\w{123\dots d}$,
  we deduce de Bruijn's formula in the odd case.
    
  \end{proof}
  
\begin{example}
de Bruijn's formula in the case $d=3$ gives
\begin{align*}
 \inv\left(
        \begin{gathered}
        \begin{ytableau}
        1  \\
        2 \\
        {3} 
        \end{ytableau}
        \end{gathered}
    \right)
  =
  \operatorname{Pf}_\shuffle
  \begin{pmatrix}
    0 & \w{12-21} & \w{13-31}   & \w1 \\
    \w{21-12} & 0 & \w{23-32}   & \w2 \\
    \w{31-13} & \w{32-23} & 0   & \w3 \\
         -\w1 &      -\w2 &-\w3 & 0
       \end{pmatrix} 
\end{align*}
whereas Theorem~\ref{thm:main} gives
\begin{align*}
 \inv\left(
        \begin{gathered}
        \begin{ytableau}
        1  \\
        2 \\
        {3} 
        \end{ytableau}
        \end{gathered}
    \right)^{\shuffle 2}
  &=
  8\
  {\det}_\shuffle 
  \begin{pmatrix}
    \w{11} & \w{12} & \w{13} \\
    \w{21} & \w{22} & \w{23} \\
    \w{31} & \w{32} & \w{33}
  \end{pmatrix},
\end{align*}
which, by the preceding argument, is equal to
\begin{align*}
  {\det}_\shuffle
  \begin{pmatrix}
    0 & \w{12-21} & \w{13-31}   & \w1 \\
    \w{21-12} & 0 & \w{23-32}   & \w2 \\
    \w{31-13} & \w{32-23} & 0   & \w3 \\
         -\w1 &      -\w2 &-\w3 & 0
  \end{pmatrix}.
\end{align*}

\end{example}

 \bibliographystyle{alpha}
 \bibliography{bibShuffles}

\newcommand{\etalchar}[1]{$^{#1}$}
\begin{thebibliography}{KSH{\etalchar{+}}17}

\bibitem[AFS19]{bib:amnd}
Carlos Am\'{e}ndola, Peter Friz, and Bernd Sturmfels.
\newblock Varieties of signature tensors.
\newblock {\em Forum Math. Sigma}, 7:e10, 54, 2019.
\newblock \url{https://doi.org/10.1017/fms.2019.3}.

\bibitem[AL04]{aguiar2004quadri}
Marcelo Aguiar and Jean-Louis Loday.
\newblock Quadri-algebras.
\newblock {\em Journal of Pure and Applied Algebra}, 191(3):205--221, 2004.

\bibitem[And83]{bib:And1883}
C~Andr{\'e}ief.
\newblock Note sur une relation entre les int{\'e}grales d{\'e}finies des
  produits des fonctions.
\newblock {\em M{\'e}m. de la Soc. Sci. Bordeaux}, 2(1):1--14, 1883.

\bibitem[Art57]{MR0082463}
Emil Artin.
\newblock {\em Geometric algebra}.
\newblock Interscience Publishers, Inc., New York-London, 1957.

\bibitem[Bou07]{Bourbaki}
Nicolas Bourbaki.
\newblock {\em Elements de Mathematique Algebre - Chapitre 9}.
\newblock Springer, 2007.

\bibitem[BR10]{burgunder2010tridendriform}
Emily Burgunder and Mar{\'\i}a Ronco.
\newblock Tridendriform structure on combinatorial hopf algebras.
\newblock {\em Journal of Algebra}, 324(10):2860--2883, 2010.

\bibitem[CGM20]{bib:colmen1}
Laura Colmenarejo, Francesco Galuppi, and Mateusz Micha\l{}ek.
\newblock Toric geometry of path signature varieties.
\newblock {\em Advances in Applied Math}, 121:102102, 2020.
\newblock Latest version available in arXiv
  \url{https://arxiv.org/pdf/1903.03779.pdf}.

\bibitem[Che57]{bib:chen}
Kuo-Tsai Chen.
\newblock Integration of paths, geometric invariants and a generalized
  {B}aker-{H}ausdorff formula.
\newblock {\em Ann. of Math. (2)}, 65:163--178, 1957.
\newblock \url{https://doi.org/10.2307/1969671}.

\bibitem[CP18]{bib:colmenPreis}
Laura Colmenarejo and Rosa Prei\ss.
\newblock Signatures of paths transformed by polynomial maps.
\newblock {\em To appear on Contributions to Algebra and Geometry in 2020},
  2018.
\newblock Available at \url{https://arxiv.org/abs/1812.05962}.

\bibitem[dB55]{bib:deBruijn}
Nicolaas~G de~Bruijn.
\newblock On some multiple integrals involving determinants.
\newblock {\em J. Indian Math. Soc. (N.S.)}, 19:133--151 (1956), 1955.

\bibitem[DR19]{bib:DR2018}
Joscha Diehl and Jeremy Reizenstein.
\newblock Invariants of multidimensional time series based on their
  iterated-integral signature.
\newblock {\em Acta Appl. Math.}, 164:83--122, 2019.
\newblock \url{https://doi.org/10.1007/s10440-018-00227-z}.

\bibitem[EML53]{bib:EM}
Samuel Eilenberg and Saunders Mac~Lane.
\newblock On the groups {$H(\Pi,n)$}. {I}.
\newblock {\em Ann. of Math. (2)}, 58:55--106, 1953.
\newblock \url{https://doi.org/10.2307/1969820}.

\bibitem[FH91]{bib:FuHa}
William Fulton and Joe Harris.
\newblock {\em Representation theory}, volume 129 of {\em Graduate Texts in
  Mathematics}.
\newblock Springer-Verlag, New York, 1991.
\newblock A first course, Readings in Mathematics.

\bibitem[FH14]{bib:FH}
Peter~K Friz and Martin Hairer.
\newblock {\em A course on rough paths}.
\newblock Universitext. Springer, Cham, 2014.
\newblock With an introduction to regularity structures.

\bibitem[For19]{forrester2019meet}
Peter~J Forrester.
\newblock Meet andr{\'e}ief, bordeaux 1886, and andreev, kharkov 1882--1883.
\newblock {\em Random Matrices: Theory and Applications}, 8(02):1930001, 2019.

\bibitem[FP13]{bib:FP}
Lo\"{\i}c Foissy and Fr\'{e}d\'{e}ric Patras.
\newblock Natural endomorphisms of shuffle algebras.
\newblock {\em Internat. J. Algebra Comput.}, 23(4):989--1009, 2013.
\newblock \url{https://doi.org/10.1142/S0218196713400183}.

\bibitem[Gal19]{bib:galup}
Francesco Galuppi.
\newblock The rough {V}eronese variety.
\newblock {\em Linear Algebra Appl.}, 583:282--299, 2019.
\newblock \url{https://doi.org/10.1016/j.laa.2019.08.029}.

\bibitem[Joh05]{johansson2005random}
Kurt Johansson.
\newblock Random matrices and determinantal processes.
\newblock In {\em Mathematical Statistical Physics, Session LXXXIII: Lecture
  Notes of the Les Houches Summer School 2005}. Citeseer, 2005.

\bibitem[KO19]{bib:KO2019}
Franz~J Kir{\'a}ly and Harald Oberhauser.
\newblock Kernels for sequentially ordered data.
\newblock {\em Journal of Machine Learning Research}, 20, 2019.

\bibitem[KSH{\etalchar{+}}17]{bib:kormilitzin2017detecting}
Andrey Kormilitzin, Kate~EA Saunders, Paul~J Harrison, John~R Geddes, and Terry
  Lyons.
\newblock Detecting early signs of depressive and manic episodes in patients
  with bipolar disorder using the signature-based model.
\newblock 2017.
\newblock Available at \url{https://arxiv.org/abs/1708.01206}.

\bibitem[Led93]{MR1221053}
Walter Ledermann.
\newblock A note on skew-symmetric determinants.
\newblock {\em Proc. Edinburgh Math. Soc. (2)}, 36(2):335--338, 1993.
\newblock \url{https://doi.org/10.1017/S0013091500018423}.

\bibitem[LNO14]{bib:lyons2014feature}
Terry Lyons, Hao Ni, and Harald Oberhauser.
\newblock A feature set for streams and an application to high-frequency
  financial tick data.
\newblock In {\em Proceedings of the 2014 International Conference on Big Data
  Science and Computing}, pages 1--8, 2014.

\bibitem[Lod95]{loday1995cup}
Jean-Louis Loday.
\newblock Cup-product for leibniz cohomology and dual leibniz algebras.
\newblock {\em Mathematica Scandinavica}, pages 189--196, 1995.

\bibitem[LQ02]{bib:LQ}
Terry Lyons and Zhongmin Qian.
\newblock {\em System control and rough paths}.
\newblock Oxford Mathematical Monographs. Oxford University Press, Oxford,
  2002.
\newblock Oxford Science Publications.

\bibitem[LT02]{bib:LuqueThibon}
Jean-Gabriel Luque and Jean-Yves Thibon.
\newblock Pfaffian and {H}afnian identities in shuffle algebras.
\newblock {\em Adv. in Appl. Math.}, 29(4):620--646, 2002.
\newblock \url{https://doi.org/10.1016/S0196-8858(02)00036-2}.

\bibitem[LZL{\etalchar{+}}19]{bib:li2019skeleton}
Chenyang Li, Xin Zhang, Lufan Liao, Lianwen Jin, and Weixin Yang.
\newblock Skeleton-based gesture recognition using several fully connected
  layers with path signature features and temporal transformer module.
\newblock In {\em Proceedings of the AAAI Conference on Artificial
  Intelligence}, volume~33, pages 8585--8593, 2019.

\bibitem[PSS19]{bib:pfeff}
Max Pfeffer, Anna Seigal, and Bernd Sturmfels.
\newblock Learning paths from signature tensors.
\newblock {\em SIAM J. Matrix Anal. Appl.}, 40(2):394--416, 2019.
\newblock \url{https://doi.org/10.1137/18M1212331}.

\bibitem[Reu93]{bib:reut}
Christophe Reutenauer.
\newblock {\em Free {L}ie algebras}, volume~7 of {\em London Mathematical
  Society Monographs. New Series}.
\newblock The Clarendon Press, Oxford University Press, New York, 1993.
\newblock Oxford Science Publications.

\bibitem[Sch58]{schutzenberger1958propriete}
Marcel~P Sch{\"u}tzenberger.
\newblock Sur une propri{\'e}t{\'e} combinatoire des algebres de lie libres
  pouvant {\^e}tre utilis{\'e}e dans un probleme de math{\'e}matiques
  appliqu{\'e}es.
\newblock {\em S{\'e}minaire Dubreil. Alg{\`e}bre et th{\'e}orie des nombres},
  12(1):1--23, 1958.

\bibitem[Sta99]{bib:Stanley}
Richard~P Stanley.
\newblock {\em Enumerative combinatorics. {V}ol. 2}, volume~62 of {\em
  Cambridge Studies in Advanced Mathematics}.
\newblock Cambridge University Press, Cambridge, 1999.

\bibitem[Stu96]{bib:stuToric}
Bernd Sturmfels.
\newblock {\em Gr\"{o}bner bases and convex polytopes}, volume~8 of {\em
  University Lecture Series}.
\newblock American Mathematical Society, Providence, RI, 1996.

\bibitem[Wey46]{bib:weyl1946classical}
Hermann Weyl.
\newblock {\em The classical groups: their invariants and representations},
  volume~45.
\newblock Princeton university press, 1946.

\end{thebibliography}
 
\end{document}